\documentclass[11pt,a4paper, parskip=half]{scrartcl} 
\usepackage[utf8]{inputenc}
\usepackage[english]{babel}
\usepackage[T1]{fontenc}
\usepackage{lmodern}
\usepackage[left=3cm,right=3cm,top=3cm,bottom=3cm]{geometry}
\usepackage[runin]{abstract}
    \abslabeldelim{.}

\usepackage{todonotes}

\usepackage{amsmath, amsfonts, amssymb, amsthm, amstext}
\usepackage{mathtools}
\usepackage{mathrsfs}
\usepackage{mathdots}
\usepackage{cases}

\usepackage{graphicx}
\usepackage{subcaption}
\usepackage{algorithm}

\font\mfett=cmmib10 at11pt
 at9pt
\def\bgamma{\hbox{\mfett\char013}}
\def\gamra{\hbox{\mfett\char013}}

\def\bphi{\hbox{\mfett\char030}}

\usepackage{enumitem}

\usepackage{hyperref} 
\usepackage[capitalise]{cleveref}
	
\newcounter{thm}
\numberwithin{thm}{section}
\numberwithin{equation}{section}

	\newtheoremstyle{myplain}		
			{}			
			{}			
			{\itshape}				
			{}				
			{\sffamily\bfseries}				
			{.}		
			{ }				
			{\thmname{#1}\thmnumber{ #2}\textnormal{\textsf{\thmnote{ (#3)}}}}			
    \newtheoremstyle{mybreak}
            {}{}{}{}{\sffamily\bfseries}{.}{\newline}
            {\thmname{#1}\thmnumber{ #2}\textnormal{\textsf{\thmnote{ (#3)}}}}
	\newtheoremstyle{mydef}
			{}{}{}{}{\sffamily\bfseries}{.}{ }
			{\thmname{#1}\thmnumber{ #2}}
	\newtheoremstyle{myrem}
			{}{}{}{}{\sffamily\itshape}{.}{ }
			{\thmname{#1}\thmnumber{ #2}}

\theoremstyle{myplain}
	\newtheorem{theorem}[thm]{Theorem}

    \newtheorem{corollary}[thm]{Corollary}
\theoremstyle{mybreak}
\theoremstyle{mydef}
	
	\newtheorem{remark}[thm]{Remark}
\theoremstyle{mydef}
	\newtheorem{example}[thm]{Example}

	\newcommand{\cc}{\mathbb{C}}
		\newcommand{\zz}{\mathbb{Z}}
		\newcommand{\nn}{\mathbb{N}}
	\newcommand{\rr}{\mathbb{R}}

\newcommand{\argmax}{\mathop{\mathrm{argmax}}}

\allowdisplaybreaks

\def\sumprime_#1^#2{
    \setbox0=\hbox{$\scriptstyle{#1}$}
    \setbox1=\hbox{$\scriptstyle{#2}$}
    \setbox2=\hbox{$\displaystyle{\sum}$}
    \setbox4=\hbox{${}^\prime\mathsurround=0pt$}
    \dimen0=.5\wd0 \advance\dimen0 by-.5\wd2
    \ifdim\dimen0>0pt
        \ifdim\dimen0>\wd4 \kern\wd4
        \else\kern\dimen0
        \ifdim\dimen1>\wd4 \kern\wd4
        \else\kern\dimen1
    \fi\fi\fi
\mathop{{\sum}^\prime}_{\kern-\wd4 #1}^{\kern-\wd4 #2}
}

\title{\Large ESPRIT versus ESPIRA for reconstruction of short cosine sums and its application}
\author{Nadiia Derevianko\footnote{Institute for Numerical and Applied Mathematics, G\"ottingen University, Lotzestr.\ 16-18, 37083 G\"ottingen, Germany, \{n.derevianko,plonka,r.razavi\}@math.uni-goettingen.de} \footnote{Corresponding author} \quad Gerlind Plonka$^{*}$ \quad Raha Razavi$^{*}$}
\date{\small\textsc{Dedicated to Claude Brezinski on the occasion of his 80th birthday}}

\begin{document}
	\let\oldproofname=\proofname
	\renewcommand{\proofname}{\itshape\sffamily{\oldproofname}}

\maketitle

\begin{abstract}
 In this paper we introduce two new  algorithms  for stable  approximation with and recovery of short  cosine sums. The used signal model contains cosine terms with arbitrary real positive frequency parameters and therefore strongly generalizes usual Fourier sums. The proposed methods both employ a set of equidistant signal values as input data.
The ESPRIT  method for cosine sums is a Prony-like method and applies matrix pencils of Toeplitz $+$ Hankel matrices while the ESPIRA method is based on rational approximation of DCT data and can be understood as a matrix pencil method for special Loewner matrices.  Compared to  known numerical methods  for recovery of exponential sums, the design of the considered new algorithms directly exploits the special real structure of the signal model  and therefore usually provides real parameter estimates for noisy input data, while  the known general recovery algorithms for complex exponential sums  tend to yield complex parameters in this case. \\[1ex]

\end{abstract}

\section{Introduction}
\label{introduction}

We consider cosine sums of the form 
\begin{equation}\label{1.1}
f(t) = \sum_{j=1}^{M} \gamma_{j} \, \cos(\phi_{j} t),
\end{equation}
where $M \in {\mathbb N}$, $\gamma_{j} \in {\mathbb R}\setminus \{ 0\}$, and  the frequency parameters $\phi_{j}  \in [0, K)$ (with $K>0$) are pairwise distinct. We define a step size $h$ with $h=\frac{\pi}{K}$ and want to study
 the following reconstruction problems:
\begin{description}
\item{1.} How to reconstruct a function $f$ in a stable way from function values $f_{k} = f(\frac{h(2k +1)}{2})$, $k=0, \ldots, N-1$, $N > 2M$ ?
\item{2.} How to reconstruct a function $f$ in a stable way from noisy function values $y_{k} = f(\frac{h(2k +1)}{2})+ \epsilon_{k}$, $k=0, \ldots, N-1$, $N > 2M$, where the noise vector $(\epsilon_{k})_{k=0}^{N-1}$ has a Gaussian or uniform distribution with mean value $0$?
\item{3.} How to approximate an even smooth function $g$ by a short cosine sum in an efficient way?
\end{description}

The recovery of cosine sums  from a finite set of possibly corrupted signal samples as well as the approximation with short cosine sums play an important role in many signal processing problems.
Applications of sparse cosine sums can be found in sparse phase retrieval \cite{Beinert17}, and for  exact approximation of Bessel functions \cite{Cu2020}. Other applications concern the recovery of cosine sums that appear in optical models as measured spectral interferograms, as for example in optical coherence tomography \cite{Izat}.

The problem is closely related with the recovery of and approximation by sums of exponentials of the form  $\sum_{j=1}^{2M} \gamma_{j} \, {\mathrm e}^{{\mathrm i} \phi_{j} t }$, 
which has been extensively studied within the last years, see e.g.\ \cite{Osborne95, Pereyra10, VMB02, BM05, PT2013, PT14, PT15, PPST18, ZP19, PPD21, DP21, DPP21}.
Moreover, there is a close connection to the question of  extrapolation of the given sequence of input values $(f_{k})_{k=0}^{N-1}$, see e.g. \cite{Bre1,Bre2,Bre3}.

At the first glance, the known reconstruction algorithms for exponential sums in \cite{DPP21, Hua90, PT2013, PT15, RK89} seem to cover also the problems raised above, since the cosine sum $f$ in (\ref{1.1}) can be simply transferred into an exponential sum $ \frac{1}{2}\sum_{j=1}^{M} \gamma_{j}({\mathrm e}^{{\mathrm i}\phi_{j}t} + {\mathrm e}^{-{\mathrm i}\phi_{j}t})$, which is just a special case of an exponential sum of length $2M$.
Indeed, in case of exact measurement data, any reconstruction algorithm for exponential sums can be directly used for the recovery of the parameters of $f$ in (\ref{1.1}).
However, if we have noisy input data, or if we want to approximate a given function $g$ by a real cosine sum, then the known algorithms for general exponential sums usually no longer provide us a real solution that can be represented in the model (\ref{1.1}), i.e., the optimal real frequency parameters $\phi_{j}$ cannot be derived from the reconstruction of the corresponding exponential sum.

There exist a few approaches that are directly concerned with the reconstruction of cosine sums, see e.g. \cite{Beinert17,PSK19,Cu2020,SP20,KP21}, which are all based on Prony's method. An explicit recovery algorithm has been only derived in \cite{Cu2020}. In \cite{PT2014}, the reconstruction of sparse Chebyshev polynomials has been considered, which can be seen as a special case of the recovery of (\ref{1.1}) if $\phi_{j}$ are restricted to the set $c \, {\mathbb N}$ with some constant $c>0$. 


The goal of this paper is to propose two different recovery algorithms for cosine sums of the type (\ref{1.1}).
In Section 2, we start with repeating a variant of Prony's method for direct recovery of cosine sums in (\ref{1.1}) for exact input data $f_{k} = f(\frac{h(2k +1)}{2})$, $k=0, \ldots, N-1$, $N > 2M$, and derive a new ESPRIT-type algorithm to achieve better numerical stability in real arithmetics.
Our new ESPRIT algorithm for cosine sums is based on a matrix pencil method for Hankel$+$Toeplitz matrices and differs from the previous methods given in \cite{PT2014} and \cite{Cu2020}.

In a recent paper \cite{DPP21}, we had proposed an ESPIRA (Estimation of Signal Parameters by Iterative Rational Approximation) algorithm  for recovery of exponential sums. As shown in \cite{DPP21}, this new method can be successfully applied for reconstruction of exponential sums from noisy data as well as for function approximation, and it essentially outperforms all previous methods in both regards. 
In Section 3, we now present a new ESPIRA algorithm, which is especially adapted to cosine sums. For this purpose, the problem of parameter reconstruction is transferred to a problem of rational interpolation of DCT (Discrete cosine transform) transformed data. The rational interpolation problem is then solved in a stable way using the recently proposed AAA algorithm, \cite{NST18}. In Section 4, we show that the ESPIRA algorithm from Section 3 can be reinterpreted as a matrix-pencil method for special Loewner matrices. This representation enables us to derive a slightly different algorithm (called ESPIRA-II) which employs this matrix pencil approach. 
Finally, in Section 5 we show at several examples that the new ESPRIT and ESPIRA algorithms work efficiently for reconstruction, but  moreover also for function recovery from noisy data and for function approximation in double precision arithmetics.

Throughout this paper, we will use the matrix notation ${\mathbf A}_{M,N}$ for real matrices of size $M \times N$ and the submatrix notation ${\mathbf A}_{M,N}(m : n, k : \ell)$ to denote a submatrix of ${\mathbf A}_{M,N}$ with rows indexed $m$ to $n$ and columns indexed $k$ to $\ell$, where (as in Matlab) the first row and first column has index $1$ (even though the row- and column indices for the definition of the matrix may start with 0). For square matrices we often use the short notation ${\mathbf A}_{M}$ instead of ${\mathbf A}_{M,M}$ .

\section{Prony's method and ESPRIT for cosine sums}

\subsection{A variant of Prony's method for cosine sums}

We briefly summarize the approaches for reconstruction of sparse cosine sums $f$ in (\ref{1.1}) using a variant of the classical Prony method, see e.g. \cite{PSK19,SP20}. We slightly modify those results  with regard to the structure of  given function values.

\begin{theorem}\label{theo1}
Let some $K>0$ be given. Assume that $f$ is of the form $(\ref{1.1})$, where $M$ is known beforehand, and where the pairwise distinct parameters $\phi_{j} \in [0, K) \subset {\mathbb R}$ and $\gamma_{j} \in {\mathbb R}\setminus \{0\}$, $j=1, \ldots, M$, are unknown.
Then $f$ can be uniquely reconstructed  from the samples $f( \frac{h(2k+1)}{2})$, $k=0, \ldots , 2M-1$, with $h=\frac{\pi}{K}$, i.e., all parameters of $f$ can be uniquely recovered.
\end{theorem}

\begin{proof}
We define the (characteristic) polynomial
$$ p(z) = \prod_{j=1}^{M} \left(z- \cos ( \phi_{j} h) \right) $$
of degree $M$, where by assumption the zeros $\cos( \phi_{j} h )$ are pairwise distinct. Then $p(z)$ can be rewritten as
\begin{equation}\label{poly}
 p(z) = \sum_{\ell=0}^{M} p_{\ell} \, T_{\ell}(z), 
 \end{equation}
where $T_{\ell}(z)$, $\ell \in {\mathbb N}_{0}$,  denotes the Chebyshev polynomial of first kind of degree $\ell$, which is for $z\in [-1,1]$ given by $T_{\ell}(z) = \cos (\ell\, \arccos z)$. For $\ell \ge 1$  the leading  coefficient of $T_{\ell}$ is  $2^{\ell-1}$, therefore it follows that $p_{M} = 2^{-M+1}$. 

Observe that $f$ in (\ref{1.1}) is an even function, therefore we have 
$$ \textstyle f_{k} :=f\left(\frac{h(2k+1)}{2}\right) = f\left(-\frac{h(2k+1)}{2}\right) = f_{-k-1}, \qquad k=0, \ldots , 2M-1. $$
The coefficients $p_{\ell}$ of the polynomial $p(z)$ in (\ref{poly}) satisfy for $m=0, \ldots , M-1$, the equations
\begin{align}
\nonumber
 \sum_{\ell=0}^{M} p_{\ell} (f_{m+\ell} + f_{m-\ell}) 
 =&  \sum_{\ell=0}^{M} p_{\ell} \, \sum_{j=1}^{M} \textstyle \gamma_{j} \left( \cos\left(\left(\frac{\phi_{j} h (1+2(m+\ell))}{2}\right)\right)  + \cos\left( \left(\frac{\phi_{j} h(1+2(m-\ell))}{2}\right) \right) \right) \\
 \nonumber
  =&  \sum_{\ell=0}^{M} p_{\ell} \, \sum_{j=1}^{M} \textstyle \gamma_{j} \, 2 \, \textstyle \left( \cos(  (\frac{\phi_{j} h(1+2m)}{2})) \, \cos (\phi_{j} h \ell) \right) \\
  \nonumber
=&  2  \sum_{j=1}^{M} \gamma_{j} \, {\textstyle{\cos\left(  \left(\frac{\phi_{j} h(1+2m)}{2}\right)\right)}} \sum_{\ell=0}^{M} p_{\ell} \, \textstyle \cos(\phi_{j} h \ell) \\
\label{diff}
=& 2  \sum_{j=1}^{M} \gamma_{j} \, {\textstyle{\cos\left(  \left(\frac{\phi_{j} h(1+2m)}{2}\right)\right)}} \sum_{\ell=0}^{M} \textstyle p_{\ell} \, T_{\ell}(\cos (\phi_{j} h)) =0.
\end{align}
Using the known function values $f_{k}$, $k=-2M, \ldots , 2M-1$,  we define the $M \times M$-matrix
$${\mathbf M}_{M} := \left( f_{m+\ell} + f_{m - \ell}
 \right)_{m,\ell=0}^{M-1}. $$
Further, let ${\mathbf p} :=(p_{0}, \ldots , p_{M-1})^{T}$ be the vector of polynomial coefficients in (\ref{poly}). Then, with $p_{M}= 2^{-M+1}$, (\ref{diff}) yields  the equation system 
\begin{equation} \label{sys1}
{\mathbf M}_{M} \, {\mathbf p} = - 2^{-M+1} \, \left( f_{m+M} + f_{m-M} 
\right)_{m=0}^{M-1}. 
\end{equation}
The matrix ${\mathbf M}_{M}$ has Toeplitz$+$Hankel structure. It is invertible, since we obtain the factorization
\begin{align*}
{\mathbf M}_{M}  &= \left( \sum_{j=1}^{M} \gamma_{j} \textstyle \left(  \cos \left(\phi_{j} h \left(\frac{1+2(m+\ell)}{2}\right)\right)  + \cos\left(\phi_{j} h \left(\frac{1+2(m-\ell)}{2}\right)\right) \right) \right)_{m,\ell=0}^{M-1} \\
&= 2 \left( \sum_{j=1}^{M} \gamma_{j}  \, \textstyle \left( \cos \left( \phi_{j} h \left(\frac{1+2m}{2}\right) \right) \, \cos( \phi_{j} h \ell) \right) \right)_{m,\ell=0}^{M-1} \\
&= 2 \, {\mathbf V}_{M}^{(1)} \, \textrm{diag} (\gamma_{j})_{j=1}^{M} \, ({\mathbf V}_{M}^{(2)})^{T} 
\end{align*}
with the generalized Vandermonde matrices 
$$ \textstyle{\mathbf V}_{M}^{(1)} = \left( \cos \left(  \left(\frac{\phi_{j} h(1+2m)}{2} \right) \right) \right)_{m=0,j=1}^{M-1,M}, \qquad 
 {\mathbf V}_{M}^{(2)} = \left(\cos(\phi_{j} h \ell) \right)_{\ell=0,j=1}^{M-1,M}.
$$
The two Vandermonde matrices  ${\mathbf V}_{M}^{(1)}$ and ${\mathbf V}_{M}^{(2)}$ as well as the diagonal matrix are invertible by assumption. Therefore, $p(z)$ is uniquely defined by (\ref{sys1}) and its zeros $\cos(\phi_{j} h)$ are uniquely determined. Having $\cos(\phi_{j} h)$, frequencies $\phi_j$ can be unambiguously extracted if $0\leq \phi_j h \leq \pi$, which gives us the restriction $\phi_j \in [0,K]$.
Finally, the coefficients $\gamma_{j}$ of $f$ in (\ref{1.1}) are determined by the linear system 
$$ f_{k}  = \sum_{j=1}^{M} \gamma_{j} \textstyle \, \cos \left( \frac{\phi_{j} h(2k+1)}{2}\right), \ k=0,\ldots, 2M-1.
$$
\end{proof}

\subsection{ESPRIT for sparse cosine sums}
While the reconstruction of sparse cosine sums can be theoretically performed according to the constructive method described in the proof of Theorem \ref{theo1}, this procedure is numerically not stable. We are particularly interested in the recovery of cosine sums from noisy data and in function  approximation by short cosine sums. Therefore we need a method which is able to estimate also the number $M$ of terms and provides a good function approximation also in the case of noisy data.  
In this section, we will transfer the well-known ESPRIT method \cite{RK89} to our setting. For a related approach, we refer to \cite{PT2014}, where such a method has been derived for the reconstruction of sparse Chebyshev polynomials, and to \cite{KP21}, where the ESPRIT method is considered for the generalized Prony method. Further, in \cite{Cu2020}, an ESPRIT-like method has been employed for other input data, namely $f(j\Delta)$ with $\Delta=\frac{B}{2M-1}$ or $\Delta=\frac{B}{2M}$  for approximation of $f$ by a cosine sum with $M$ terms in $[0,B]$.
We remark that all earlier approaches differ from the ESPRIT method that we will present here.

\medskip

Let us assume that $L$ is a given upper bound of the sparsity $M$ in (\ref{1.1}), and let $N$  be a sufficiently large number of given  function samples $f_{k} = f\left(\frac{h(1+2k)}{2}\right)$, $k=-N, \ldots , N-1$,  such that $M \le L \le N/2$. 
We consider now the following three Toeplitz$+$Hankel matrices  of size $(N-L,L)$, 
\begin{align*}
{\mathbf M}_{N-L,L}^{(-1)} &:= \textstyle \frac{1}{2}  \left(  f_{m+\ell-1} + f_{m-\ell-1} \right)_{m=0,\ell=0}^{N-L-1,L-1}, \\
{\mathbf M}_{N-L,L}^{(0)} &:= \textstyle \frac{1}{2}  \left(  f_{m+\ell} + f_{m-\ell} \right)_{m=0,\ell=0}^{N-L-1,L-1}, \\
{\mathbf M}_{N-L,L}^{(1)} &:= \textstyle \frac{1}{2} \left(  f_{m+\ell+1} + f_{m-\ell+1} \right)_{m=0,\ell=0}^{N-L-1,L-1}. \\
\end{align*}
Then, similarly as in the proof of Theorem \ref{theo1}, we observe  that these matrices  possess the representations
\begin{align*}
{\mathbf M}_{N-L,L}^{(-1)} &= \textstyle \left( \sum\limits_{j=1}^{M} \gamma_{j} \left( \cos\left( \phi_{j} h \left(\frac{-1+2m}{2}\right) \right) \, \cos ( \phi_{j} h \ell ) \right) \right)_{m=0,\ell=0}^{N-L-1,L-1}, \\
{\mathbf M}_{N-L,L}^{(0)} &= \textstyle \left( \sum\limits_{j=1}^{M} \gamma_{j} \left( \cos\left( \phi_{j} h \left(\frac{1+2m}{2}\right) \right) \, \cos ( \phi_{j} h \ell ) \right) \right)_{m=0,\ell=0}^{N-L-1,L-1}, \\
{\mathbf M}_{N-L,L}^{(1)} &= \textstyle \left( \sum\limits_{j=1}^{M} \gamma_{j} \left( \cos\left( \phi_{j} h \left(\frac{3+2m}{2}\right) \right) \, \cos (\phi_{j} h \ell ) \right) \right)_{m=0,\ell=0}^{N-L-1,L-1}. 
\end{align*}
Therefore, we conclude that 
\begin{align*}
 & \textstyle \frac{1}{2} \Big({\mathbf M}_{N-L,L}^{(-1)} + {\mathbf M}_{N-L,L}^{(1)} \Big) \\
 &= \textstyle
\left( \sum\limits_{j=1}^{M} \frac{\gamma_{j}}{2} \left( \cos\left( \phi_{j} h\left(\frac{-1+2m}{2}\right) \right) + \cos\left( \phi_{j} h \left(\frac{3+2m}{2}\right) \right) 
\right)
\, \cos ( \phi_{j} h \ell)  \right)_{m=0,\ell=0}^{N-L-1,L-1} \\
&= \textstyle \left( \sum\limits_{j=1}^{M} \gamma_{j} \left( \cos\left( \phi_{j} h \left(\frac{1+2m}{2}\right) \right)  \cos( \phi_{j} h )
\, \cos (\phi_{j} h \ell) \right) \right)_{m=0,\ell=0}^{N-L-1,L-1} \\
&=  \textstyle {\mathbf V}_{N-L,M}^{(1)} \,  \textrm{diag}  \left(\left( \gamma_{j}  \right)_{j=1}^{M} \right)  \, 
\textrm{diag} \left( \left( \cos(\phi_{j} h)\right)_{j=1}^{M}  \right)\, ({\mathbf V}_{L, M}^{(2)})^{T} 
\end{align*}
with the generalized Vandermonde matrices 
$$ \textstyle {\mathbf V}_{N-L,M}^{(1)} = \Big(\cos \Big( \phi_{j} h \Big( \frac{1+2m}{2} \Big) \Big)\Big)_{m=0,j=1}^{N-L-1,M}, \qquad 
{\mathbf V}_{L, M}^{(2)} = \Big(\cos (\phi_{j} h  \ell) \Big)_{\ell=0,j=1}^{L-1,M},
$$
with full column rank $M$, while ${\mathbf M}_{N-L,L}^{(0)}$ possesses the factorization
\begin{equation}\label{fak0} {\mathbf M}_{N-L,L}^{(0)} = {\mathbf V}_{N-L,M}^{(1)} \,  \textrm{diag} \Big( \gamma_{j} \Big)_{j=1}^{M}   \, ({\mathbf V}_{L, M}^{(2)})^{T}. 
\end{equation}
Thus, it follows that the values $\cos( \phi_{j} h)$, $j=1, \ldots , M$, are eigenvalues of the matrix pencil
\begin{equation}\label{pencil1}
 z {\mathbf M}_{N-L,L}^{(0)} - \textstyle \frac{1}{2} \Big({\mathbf M}_{N-L,L}^{(-1)} + {\mathbf M}_{N-L,L}^{(1)} \Big). 
 \end{equation}
The three matrices ${\mathbf M}_{N-L,L}^{(\kappa)}$, $\kappa=-1,0, 1$, can be all obtained as submatrices 
of 
$$ {\mathbf M}_{N-L+2,L} \coloneqq \textstyle \frac{1}{2} \left( f_{m+\ell-1} + f_{m-\ell-1}\right)_{m=0,\ell=0}^{N-L+1,L-1}, $$
where, using the Matlab notation of submatrices, we have
\begin{align*}
{\mathbf M}_{N-L,L}^{(-1)} & = {\mathbf M}_{N-L+2,L}(1\!:\!N-L, \, 1\!:\!L), \\
{\mathbf M}_{N-L,L}^{(0)} & = {\mathbf M}_{N-L+2,L}(2\!:\!N-L+1, \, 1\!:\!L), \\
{\mathbf M}_{N-L,L}^{(1)} & = {\mathbf M}_{N-L+2,L}(3\!:\!N-L+2, \, 1\!:\!L), 
\end{align*}
i.e., ${\mathbf M}_{N-L,L}^{(-1)}$ is obtained from ${\mathbf M}_{N-L+2,L}$ by removing the last two rows,
${\mathbf M}_{N-L,L}^{(0)}$ is found from ${\mathbf M}_{N-L+2,L}$ by removing the first row and the last row,
and finally ${\mathbf M}_{N-L,L}^{(1)}$ is obtained  by removing the first two rows. 

\begin{remark} Using similar ideas as in \cite{PT2014}, applied to our setting, 
the eigenvalue problem (\ref{pencil1}) can be also obtained as follows for given $M= L \le N/2$.  The roots of the polynomial $p(z)$ in (\ref{poly}) are eigenvalues of a Chebyshev companion matrix $C_M(p) \in \rr^{M\times M}$ defined by
\begin{equation}\label{compcheb}
{\mathbf C}_M(p)=
\left( \begin{array}{cccccc}
0 & \frac{1}{2} & 0 & \ldots & 0 & -\frac{p_0}{2p_M} \\
1 & 0 &  \frac{1}{2} &   \ldots &  0 & -\frac{p_1}{2p_M}  \\
0 & \frac{1}{2}  & 0 & \ldots   & 0  & -\frac{p_2}{2p_M}  \\
\vdots & & & \ddots  & & \\
0 &  0 &  \ldots & \frac{1}{2}  & 0  & -\frac{p_{M-2}}{2p_M}+\frac{1}{2} \\
0 & 0 & 0 & \ldots  & \frac{1}{2}  & -\frac{p_{M-1}}{2p_M}\end{array} \right),
\end{equation}
see for example \cite{Boyd13}, i.e. 
\begin{equation}\label{det}
\mathrm{det} (z \mathbf{I}_M- {\mathbf C}_M(p))= \textstyle \frac{1}{2^{M-1}} p(z).
\end{equation}
Then (\ref{diff}) implies that 
\begin{equation}\label{mateq}
{\mathbf M}_{M}^{(0)} {\mathbf C}_M(p)= \textstyle \frac{1}{2} \Big({\mathbf M}_{M}^{(-1)} + {\mathbf M}_{M}^{(1)} \Big). 
\end{equation}
Since ${\mathbf M}_{M}^{(0)}$ has full rank $M$, it follows that  the matrix pencil
$${\mathbf M}_{M}^{(0)} ( z {\mathbf I}_{M} - {\mathbf C}_{M}(p)) 
= z {\mathbf M}_{M}^{(0)} - \textstyle \frac{1}{2} \Big({\mathbf M}_{M,M}^{(-1)} + {\mathbf M}_{M,M}^{(1)} \Big)
$$
possesses the roots of polynomial $p(z)$, i.e., $\cos(\phi_j h )$ for $j=1,\ldots, M$, as eigenvalues. 
If $M$ is unknown and we have $L$ as an upper bound for $M$ such that   $M\leq L \leq N/2$, we can modify the procedure to get the rectangular matrix pencil (\ref{pencil1}).
\end{remark}

To derive a more stable representation of the matrix pencil problem in (\ref{pencil1})
we proceed similarly as in \cite{PT2013} for exponentials sums. We employ the SVD of ${\mathbf M}_{N-L+2,L}$,
\begin{equation}\label{SVD} {\mathbf M}_{N-L+2,L} = {\mathbf U}_{N-L+2} \, {\mathbf D}_{N-L+2,L} \, {\mathbf W}_{L}, 
\end{equation}
where ${\mathbf U}_{N-L+2}$ and ${\mathbf W}_{L}$ are orthogonal square matrices of size $N-L+2$ and $L$ respectively, and 
$$ {\mathbf D}_{N-L+2,L} = \left( \begin{array}{c} \mathrm{diag} \, (\sigma_{\ell})_{\ell=1}^{L} \\ {\mathbf 0}_{N-2L+2,L} \end{array} \right) $$
contains the singular values $\sigma_{\ell}$ of ${\mathbf M}_{N-L+2,L}$, where ${\mathbf 0}_{N-2L+2,L}$ is a zero matrix of the given size.
Since ${\mathbf M}_{N-L+2,L}$ possesses a factorization of the form 
$$ {\mathbf M}_{N-L+2,L} = {\mathbf V}_{N-L+2,M}^{(1)} \,  \textrm{diag} \Big( \gamma_{j} \Big)_{j=1}^{M}   \, ({\mathbf V}_{L, M}^{(2)})^{T}, $$
similarly to (\ref{fak0}), it follows that only $M$ singular values $\sigma_{j}$ in ${\mathbf D}_{N-L+2,L}$ are non-zero, and $M$ can be found as the numerical rank of ${\mathbf M}_{N-L+2,L}$. 
The SVD in (\ref{SVD}) also yields similar factorizations for the three submatrices. Using the short notation 
$${\mathbf U}_{N-L,N-L+2}^{(\kappa)} = {\mathbf U}_{N-L+2}(\kappa+2\!:\!N-L+\kappa+1,1\!:\!N-L+2), \qquad  \kappa=-1,0,1, $$
 we find
$$ {\mathbf M}_{N-L,N-L+2}^{(\kappa)} = {\mathbf U}_{
N-L,N-L+2}^{(\kappa)} \, {\mathbf D}_{N-L+2,L} \, {\mathbf W}_{L}, \qquad  \kappa=-1,0,1. $$
By multiplication with ${\mathbf W}_{L}^{-1} = {\mathbf W}_{L}^{T}$ from the right, the matrix pencil in (\ref{pencil1}) can therefore be rewritten as 
\begin{align*}  &  \Big(z {\mathbf U}_{N-L,N-L+2}^{(0)} 
-  \textstyle \frac{1}{2} \Big( {\mathbf U}_{N-L,N-L+2}^{(-1)} + {\mathbf U}_{N-L,N-L+2}^{(1)}  \Big) \Big)
{\mathbf D}_{N-L+2,L}.
\end{align*}
Finally, a multiplication with the pseudo inverse of ${\mathbf D}_{N-L+2,L}$ yields
$$ z  {\mathbf U}_{N-L,M}^{(0)} -  \textstyle \frac{1}{2} \Big( {\mathbf U}_{N-L,M}^{(-1)} + {\mathbf U}_{N-L,M}^{(1)} \Big), $$
where we removed the zero columns. This matrix pencil problem is equivalent to the problem to find the eigenvalues $2z$ of 
$$({\mathbf U}_{N-L,M}^{(0)})^{\dagger} \Big( {\mathbf U}_{N-L,M}^{(-1)} + {\mathbf U}_{N-L,M}^{(1)}  \Big), $$
where $({\mathbf U}_{N-L,M}^{(0)})^{\dagger}$ denotes the Moore Penrose inverse of ${\mathbf U}_{N-L,M}^{(0)}$. The coefficient vector  $\bgamma=(\gamma_j)_{j=1}^{M}$ of $f$ in (\ref{1.1}) is found as solution of the linear system 
$$
{\mathbf V}_{N,M}  \, \bgamma = \textstyle \left(f\left(\frac{h(2k+1)}{2}\right)\right)_{k=0}^{N-1},
$$
with the generalized Vandermonde matrix
\begin{equation}\label{vanmat}
{\mathbf V}_{N,M}= \textstyle \left( \cos \left( \phi_{j} h \Big(\frac{2k+1}{2}\Big)\right) \right)_{k=0,j=1}^{N-1,M}.
\end{equation}
The corresponding ESPRIT algorithm is summarized in Algorithm \ref{algESPRIT}.

\begin{algorithm}[ht]\caption{ESPRIT algorithm for recovery of cosine sums}
\label{algESPRIT}
\small{
\textbf{Input:} ${\mathbf f} = (f_{k})_{k=0}^{N-1} = \big(f\left(\frac{h(2k+1)}{2}\right)\big)_{k=0}^{N-1}$ (equidistant sampling values of $f$ in (\ref{1.1}))\\
for exact sampling data (reconstruction): $L$ (upper bound for $M$  with $ L \leq  N/2$), accuracy $\epsilon >0$,\\
for noisy sampling data (approximation): $M$ (wanted length of the cosine sum), $L=N/2 > M$.

\begin{description}
\item{1.} Build the Toeplitz$+$Hankel matrix ${\mathbf M}_{N-L+2, L} \coloneqq \left(\frac{1}{2}(f_{\ell+m-1} + f_{m-\ell-1}) \right)_{\ell,m=0}^{N-L+1,L-1}$ and 
compute the SVD  of ${\mathbf M}_{N-L+2, L}$ as in (\ref{SVD}). In case of exact input data and unknown $M$, 
determine the numerical rank $M$ of ${\mathbf M}_{N-L+2, L}$ by taking the smallest $M$ such that 
$ \sigma_{M+1} < \epsilon \, \sigma_{1}$, where $\sigma_{j}$ are the ordered diagonal entries of ${\mathbf D}_{N-L+2,L}$ with $\sigma_{1} \ge \sigma_{2} \ge \ldots \geq \sigma_{L}$.
\item{2.} Form
 ${\mathbf U}_{N-L,M}^{(-1)} \coloneqq {\mathbf U}_{N-L+2}(1\!:\!N-L,1\!:\!M)$, 
  ${\mathbf U}_{N-L,M}^{(0)} \coloneqq {\mathbf U}_{N-L+2}(2\!:\!N-L+1,1\!:\!M)$,
    ${\mathbf U}_{N-L,M}^{(1)} \coloneqq {\mathbf U}_{N-L+2}(3\!:\!N-L+2,1\!:\!M)$,
and determine the vector of eigenvalues ${\mathbf z} =(z_{1}, \ldots , z_{M})^{T}$ of 
$\left(  {\mathbf U}_{N-L,M}^{(0)} \right)^{\dagger} \left( {\mathbf U}_{N-L,M}^{(-1)} + {\mathbf U}_{N-L,M}^{(1)} \right), $
where $\left({\mathbf U}_{N-L,M}^{(0)} \right)^{\dagger}$ denotes the Moore-Penrose inverse of  ${\mathbf U}_{N-L,M}^{(0)}$.
Extract $\left( \phi_{j} \right)_{j=1}^{M}$ from  $\frac{1}{2} {\mathbf z} = \left(\cos (  \phi_{j} h) \right)_{j=1}^{M}$.
\item{3.} Compute ${\bgamma} = (\gamma_{j})_{j=1}^{M}$  as the least squares solution of the linear system 
$$
{\mathbf V}_{N,M}  \, \bgamma = {\mathbf f},
$$
with the generalized Vandermonde matrix (\ref{vanmat}).
\end{description}

\noindent
\textbf{Output:} $M \in {\mathbb N}$,  $\phi_{j}, \, \gamma_{j} \in {\mathbb R}$, $j=1, \ldots , M$.}
\end{algorithm}

\begin{remark}\label{compl} \hfil \\
1. The arithmetical complexity of the SVD decomposition of the Toeplitz$+$Hankel matrix  ${\mathbf M}_{N-L+2,L+1}$ in step 1 of Algorithm \ref{algESPRIT} requires ${\mathcal O}\big((N-L)L^{2}\big)$ operations.
Step 2 involves besides the matrix inversion and matrix multiplication the solution of the eigenvalue problem for an $(N-L) \times (N-L)$ matrix with ${\mathcal O}((N-L)^{3})$ operations.
Thus, we have overall computational costs of  ${\mathcal O}((N-L)^{3})$, and for $L \approx N/2$ we require ${\mathcal O}(N^{3})$ operations.
The computational costs can be reduced by employing a partial SVD.\\
2. Instead of applying an SVD of the matrix ${\mathbf M}_{N-L+2,L+1}$ in the first step of the algorithm \ref{algESPRIT}, we can use also the QR decomposition of this matrix to improve numerical stability. This approach has been also employed for exponential sums, see \cite{Hua90,PT2013,DPP21} and is called matrix pencil method (MPM).
\end{remark}

\section{ESPIRA for reconstruction of cosine sums}

\subsection{ESPIRA-I based on rational approximation}\label{sec31}
In this section, we will derive a new algorithm for the recovery of sparse cosine sums $f$ as in (\ref{1.1}).
As before,  we assume that we are given the (possibly corrupted) samples
$$
f_{\ell}= \textstyle f\left(\frac{h (2\ell+1)}{2}\right) + \epsilon_{\ell}, \qquad  \ell=0,\ldots,N-1,
$$
with  $N> 2M$, $h=\frac{\pi}{K}$, $K>0$. Our goal is to recover  all parameters $M  \in {\mathbb N}$, $\gamma_{j} \in {\mathbb R}\setminus \{0\}$, $\phi_{j} \in [0, K)$, $j=1, \ldots , M$, of $f$ in (\ref{1.1}).

We recall that the discrete cosine transform of type II (DCT-II) is given as a matrix vector product $\hat{\mathbf f} = {\mathbf C}_{N}^{II} \, {\mathbf  f}$ with the  cosine matrix 
\begin{equation}\label{CC} 
{\mathbf C}_{N}^{II} \coloneqq  \left( \cos \left( \frac{k(2\ell+1) \pi}{2N}  \right)\right)_{k,\ell=0}^{N-1}.
\end{equation}
Then  
$$ \textstyle \sqrt{\frac{2}{N}} \, \mathrm{diag} \,  (\frac{1}{\sqrt{2}}, 1, \ldots , 1) \, {\mathbf C}_{N}^{II} $$
is an orthogonal matrix.
The DCT-II transform  of length $N$ can be performed with a fast and numerically stable algorithm with complexity ${\mathcal O}(N \log N)$  see e.g. \cite{PT05} or \cite{PPST18}, Section 6.3. 
 Using the structure of $f$ in (\ref{1.1}) (assuming that we have exact samples), we obtain for  the components $\hat{f}_{k}$, $k=0,\ldots, N-1$,
\begin{align}\nonumber
 \hat{f_{k}} &= \textstyle  \sum\limits_{\ell=0}^{N-1} f_{\ell} \, \cos\left(\frac{\pi k(2\ell+1)}{2N} \right) \\
 \nonumber
&=  \textstyle \sum\limits_{\ell=0}^{N-1} \left( \sum\limits_{j=1}^{M} \gamma_{j} \, \cos \left( \frac{\phi_{j} h (2\ell+1)}{2}\right) \right) \, \cos\left(\frac{\pi k(2\ell+1)}{2N} \right) \\
\label{fk11}
&=  \textstyle \sum\limits_{j=1}^{M} \gamma_{j} \sum\limits_{\ell=0}^{N-1} \frac{1}{2} \left( \cos \left((2\ell +1) \Big(\frac{\phi_{j} h}{2}+ \frac{k \pi}{2N}\Big)\right) + \cos \left( (2\ell +1) \Big(\frac{\phi_{j} h}{2}- \frac{k \pi}{2N}\Big) \right) \right).
\end{align}
Recall that 
\begin{equation}\label{sum1}
\textstyle \sum\limits_{\ell=0}^{N-1} \cos \left( {(2\ell+1)x}\right) = \frac{\sin (2N \,  x)}{2 \, \sin x}, \qquad x \not\in \pi {\mathbb Z}. 
\end{equation} 
We assume first that  $\phi_{j} h \not\in \frac{\pi}{N} {\mathbb Z}$ for $j=1, \ldots , M$, such that for all $k \in {\mathbb Z}$ we can apply this formula with $x= \Big(\frac{\phi_{j} h}{2} \pm \frac{k \pi}{2N}\Big)  \not\in \pi {\mathbb Z}$,
\begin{align}\label{3new}
 \hat{f_{k}} &=  \sum\limits_{j=1}^{M} \frac{\gamma_{j}}{4} \left( \frac{\sin  (\phi_{j} hN +k\pi)}{\sin\big(\frac{\phi_{j} h}{2} + \frac{k \pi}{2N}\big)} + \frac{\sin (\phi_{j} hN -k\pi)}{\sin\big(\frac{\phi_{j} h}{2} - \frac{k \pi}{2N}\big) } \right).
\end{align}
Now, observe that $\sin(\phi_{j} hN +k\pi) = \sin(\phi_{j} hN -k\pi) = (-1)^{k} \, \sin(\phi_{j} hN)$. Therefore,
\begin{align}
\nonumber
 \hat{f_{k}} &=   (-1)^{k} \sum\limits_{j=1}^{M} \frac{\gamma_{j}}{4} \, \sin(\phi_{j} hN) \,\left( \frac{\sin\big(\frac{\phi_{j} h}{2} - \frac{k \pi}{2N}\big) + \sin\big(\frac{\phi_{j} h}{2} + \frac{k \pi}{2N}\big) }{\sin\big(\frac{\phi_{j} h}{2} + \frac{k \pi}{2N}\big) \, \sin\big(\frac{\phi_{j} h}{2} - \frac{k \pi}{2N}\big)} \right)\\
 \nonumber
&=   (-1)^{k} \sum\limits_{j=1}^{M} \gamma_{j} \, \sin(\phi_{j}h N) \,  \left( \frac{\sin \big( \frac{\phi_{j} h}{2} \big) \, \cos \big(\frac{\pi k}{2N}\big)}{\cos \big(\frac{k\pi}{N}\big) - \cos \big(\phi_{j} h \big)}
 \right) \\
 \label{fk}
&=   (-1)^{k} \, \cos\left(\frac{\pi k}{2N}\right) \sum\limits_{j=1}^{M}   \left( \frac{\gamma_{j} \, \sin(\phi_{j}h N) \,\sin \big( \frac{\phi_{j} h}{2} \big) }{\cos \big(\frac{k\pi}{N}\big) - \cos \big(\phi_{j} h\big)}
 \right) .
\end{align}
The representation for $\widehat{f}_k$ in (\ref{fk}) is well-defined  if $\phi_j h N \not\in \pi \zz$ for $j=1,\ldots,M$. For $\phi_j \in \frac{\pi}{hN}\zz$ and $k=1, \ldots , N-1$, the rule of L’Hospital leads to
\begin{equation}\label{lop} 
\lim\limits_{ \phi_{j} \to \frac{\pi k}{hN} }\frac{  \sin (\phi_j hN) \sin\big( \frac{\phi_j h}{2} \big)  }{ \cos\big( \frac{\pi k}{N} \big)- \cos\big( \phi_j h  \big)  } = \frac{N \cos (\pi k) \sin\big( \frac{\pi k}{2N} \big)}{\sin\big( \frac{\pi k}{N} \big)}=\frac{(-1)^{k} N }{2 \cos\big( \frac{\pi k}{2N} \big)}.
\end{equation}

For $\phi_j \not\in \frac{\pi}{hN}\zz$ for $j=1, \ldots , M$, formula (\ref{fk}) implies
\begin{equation}\label{dctfor} \textstyle
(-1)^{k} 
 \left( \cos\big( \frac{\pi k}{2N}  \big) \right)^{-1} \widehat{f}_k =   \sum\limits_{j=1}^{M}   \left( \frac{\gamma_{j} \, \sin(\phi_{j}h N) \,\sin \big( \frac{\phi_{j} h}{2} \big) }{\cos \big(\frac{k\pi}{N}\big) - \cos \big(\phi_{j} h\big)}
 \right), \qquad k=0, \ldots , N-1.
\end{equation}
Therefore the problem of reconstruction of all parameters of the representation (\ref{1.1}) can be reformulated as  rational interpolation  problem.
We define now a rational function of type $(M-1,M)$,
\begin{equation}\label{rat}
r_M(z) \coloneqq  \textstyle \sum\limits_{j=1}^{M} \frac{a_j}{z-b_j} 
\end{equation}
with
\begin{align*}
a_j & :=  \textstyle  \gamma_j \sin \left( \frac{\phi_j h}{2} \right) \sin (\phi_j h N), \\
b_j &:= \textstyle \cos\big(  \phi_j h \big) .
\end{align*}
Then by (\ref{dctfor}), this rational function satisfies the $N$ interpolation conditions
\begin{equation}\label{intp} \textstyle
r_M \left( \cos\big( \frac{\pi k}{N} \big)  \right) =(-1)^{k}
 \left( \cos\big( \frac{\pi k}{2N}  \big) \right)^{-1} \widehat{f}_k, \qquad k=0, \ldots, N-1 .
\end{equation}

The recovery of all parameters of $f$ in (\ref{1.1}) can now be achieved as follows. First we compute a rational function $r_{M'}$ of minimal type $(M'-1,M')$ that satisfies all interpolation conditions (\ref{intp}). Then we have to rewrite this rational function as a partial  fraction decomposition as in (\ref{rat}), i.e., we determine the parameters $a_{j}$, $b_{j}$, $j=1, \ldots , M'$. The type of the rational function $r_{M'}$  provides us the number $M=M'$.
The values $b_{j} =\cos(\phi_{j} h)$, $j=1, \ldots , M$, are the poles of $r_{M}$, and we can extract $\phi_{j}$ by taking
$$ \phi_j= \textstyle \frac{\arccos b_j}{h}.$$
Once the $\phi_{j}$'s are known, the parameters $\gamma_{j}$ can be simply recovered from $a_{j}$.

We will use the AAA algorithm proposed in \cite{NST18} to evaluate the rational function $r_{M}$. The main steps of ESPIRA-I for cosine sums are described in Algorithm \ref{alg1}.

\begin{algorithm}[ht]\caption{ESPIRA-I for cosine sums}
\label{alg1}
\small{
\textbf{Input:} ${\mathbf f} = (f_{\ell})_{\ell=0}^{N-1} = \left(f\left(\frac{h(2\ell+1)}{2}\right)\right)_{\ell=0}^{N-1}$ (equidistant sampling values of $f$ in (\ref{1.1}))\\
for exact sampling data (reconstruction): $jmax =\lfloor N/2 \rfloor-1$ (upper bound for $M$),\\
\phantom{for exact sampling data (reconstruction):} ${tol}>0$ tolerance for the approximation error\\
for noisy sampling data (approximation): $jmax= M+1$ ($M$ wanted length of cosine sum).

\begin{enumerate}
\item Compute the DCT-II-vector $\hat{\mathbf f}= (\hat{f}_{k})_{k=0}^{N-1}$ with $\hat{f}_{k} = \sum\limits_{\ell=0}^{N-1} f_{\ell} \, \cos \left( \frac{\pi (2\ell+1)k}{2N}  \right)$ of ${\mathbf f}$ and
put $g_{k}:= (-1)^{k} 
 \left( \cos\left( \frac{\pi k}{2N}  \right) \right)^{-1} \widehat{f}_k$, $k=0, \ldots , N-1$.
\item Use the AAA Algorithm \ref{alg2} with at most $jmax$ iteration steps to compute a rational function $r_{M}(z)$, where for exact input data $M$ is taken as of smallest positive integer such that 
 $$\left |r_M \left( \cos\left( \frac{\pi k}{N} \right)  \right) - g_{k} \right| < tol, \qquad k=0, \ldots , N-1.$$ 
\item Use Algorithm \ref{alg3} to compute a fractional decomposition representation of $r_{M}(z)$, 
$$ r_{M}(z) = \sum_{j=1}^{M} \frac{a_{j}}{z - b_{j}}, $$
i.e., compute $a_{j}$, $b_{j}$, $j=1, \ldots , M$.
\item Set $\phi_j=\frac{\arccos b_j}{h}$ and $\gamma_j=\frac{a_j}{\sin\left( \frac{ \phi_{j} h}{2}  \right) \, \sin \left(\phi_{j} hN  \right)}$ for $j=1,\ldots,M$.
\end{enumerate}

\noindent
\textbf{Output:} $M$, $\gamma_{j}$, $\phi_{j}$, for $j=1, \ldots , M$ (all parameters of $f$ in (\ref{1.1}).}
\end{algorithm}

The computation of the DCT-II in Step 1 of Algorithm \ref{alg1} requires ${\mathcal O}(N\, \log N)$ operations, see \cite{PT05}. The complexity of the AAA Algorithm \ref{alg2} is  ${\mathcal O}(N\, M^{3})$, since Algorithm \ref{alg2} involves SVDs of Loewner matrices of dimension $(N-j) \times j$ for $j=1, \ldots, M+1$, \cite{NST18}. Finally, Algorithm \ref{alg3} involves an eigenvalue problem  with complexity ${\mathcal O}(M^{3})$. Therefore, the overall computational costs of Algorithm \ref{alg1} are ${\mathcal O}(N\, (M^{3}+\log N))$, which is very reasonable for small $M$. 

As we will outline in the last section on numerical experiments, for recovery of cosine sums from a large number of noisy input data, we will slightly change the algorithm  by taking only the function values $g_{k}$ for $k=0, \ldots \frac{N}{2}-1$ in step 2 of the algorithm. In this  way we improve the stability of the algorithm by avoiding amplification of the error by the factor $\cos(\frac{\pi k}{2N})^{-1}$ in  the definition of $g_{k}$.

\subsection{The AAA Algorithm for rational approximation}
\label{aaa-algorithm}

We will employ the recently proposed AAA algorithm \cite{NST18} for rational approximation in order to perform step 2 in Algorithm \ref{alg1}. Therefore, we shortly summarize this algorithm in our setting.
For more detailed information we refer to \cite{NST18,DHT14} or to \cite{DPP21}, where we have applied this algorithm for the recovery of complex exponential sums.
The AAA algorithm is numerically stable due to an iterative procedure using adaptively chosen interpolation sets and a barycentric  representation  of the rational interpolant.

\medskip

Let $I \coloneqq \{0, \ldots , N-1\}$ be the index set, $Z =\{ z_{k} \coloneqq \cos\left( \frac{\pi k}{N} \right): \,  k \in I \}$  the set of support points, and $ G:=\left\{ g_{k} \coloneqq (-1)^{k}  \left( \cos\left( \frac{\pi k}{2N}  \right) \right)^{-1} \widehat{f}_k,  \,  {k \in I} \right\}$ the corresponding set of known function values.
Then the AAA algorithm will find a rational function $\tilde{r}_{M}(z)$ in barycentric representation of type $(M, M)$ such that 
\begin{equation}\label{intp1}
\tilde{r}_M \left( z_{k}   \right) = g_{k}, \qquad k \in I,
\end{equation}
where $N/2$ is an upper bound of the unknown degree $M$.
In our application, we indeed need a rational function of type $(M-1, M)$ instead of $(M,M)$. This will be forced using the  model (\ref{rat})  to transfer $\tilde{r}_{M}$ into a partial fraction decomposition  in Algorithm \ref{alg3}. In case of exact data the resulting function $\tilde{r}_M$ will have a type $(M-1,M)$ (see Corollaries \ref{AAAconv} and \ref{corconv}).

We briefly describe the iteration steps of the AAA algorithm.

We initialize the sets $S_{0} := \emptyset$ and $\Gamma_{0} := I$.
In  step $J=1$ we start with $\tilde{r}_{0}(z) := g_{k_{1}}$, where $k_{1} := \argmax_{k\in I} |g_{k}|$, and set $S_{1}:=\{ k_{1}\}$, $\Gamma_{1}:=I \setminus \{k_{1}\}$.
At the iteration step $J > 1$, we proceed as follows to compute a rational function $\tilde{r}_{J-1}$ of type $(J-1,J-1)$, see also \cite{NST18,DP21,PPD21}.
Let $S_{J} \cup \Gamma_{J} = I$ be the partition of index sets found in the $(J-1)$-th iteration step, 
 with $|S_{J}| = J$ and $|\Gamma_{J}| = N-J$.
For the rational function $\tilde{r}_{J-1}$ of type $(J-1,J-1)$ we employ the barycentric ansatz
\begin{equation}\label{bar-form}
\tilde{r}_{J-1}(z) = 
\frac{\tilde{p}_{J-1}(z)}{\tilde{q}_{J-1}(z)} \coloneqq \frac{\sum\limits_{k \in S_{J}} \frac{w_k \,g_k}{z- z_{k} }}{\sum\limits_{k \in S_{J}} \frac{w_k}{z-  z_{k}  }}, 
\end{equation}
where
$w_k \in {\mathbb R}$, $k \in S_{J}$, are weights. Then we already have by construction $\tilde{r}_{J-1}(z_{k}) = g_{k}$ for all $k \in S_{J}$ if $w_{k} \neq 0$.
The weight vector ${\mathbf w}_{J} = {\mathbf w} \coloneqq \left( w_k \right)_{k \in S_{J}}$ is now chosen such that $\tilde{r}_{J-1}(z)$ approximates the remaining data $g_{\ell}$, $\ell \in \Gamma_{J}$. Further, we assume that  $\| {\mathbf w} \|_2^2 = \sum_{k \in S_{J}} w_k^2 = 1$.
To compute ${\mathbf w}$, we consider the restricted least-squares problem obtained by linearizing the interpolation conditions for $\ell \in \Gamma_{J}$,  
\begin{equation}\label{mini}
\min_{\mathbf w} \sum\limits_{\ell \in \Gamma_{J}} \left|g_{\ell} \, \tilde{q}_{J-1}\left(z_{\ell}  \right)-\tilde{p}_{J-1}\left( z_{\ell} \right) \right|^2, \quad \textrm{such~that} \quad  \| {\mathbf w} \|_2^2 = 1.
\end{equation}
We define the Loewner matrix
$$ 
{\mathbf L}_{N-J,J} \coloneqq \left( \frac{g_{\ell} - g_{k}}{z_{\ell} - z_{k}  } \right)_{\ell \in \Gamma_{J}, k \in S_{J}},
$$
and rewrite the term in (\ref{mini}) as
$$
\sum_{\ell \in \Gamma_{J}} \left|g_{\ell} \, \tilde{q}_{J-1}\left( z_{\ell}  \right)-\tilde{p}_{J-1}\left( z_{\ell} \right)\right|^2
= \sum_{\ell \in \Gamma_{J}} \left|       {{\mathbf w}}^T \, \left( 
\frac{g_{\ell} - g_{k}}{z_{\ell}- z_{k}   }\right)_{k \in S_{J}} \right|^2 
=\| {\mathbf L}_{N-J,J} {{\mathbf w}} \|_2^2. 
$$
Thus,  the minimization problem in (\ref{mini}) takes the form
$
\min_{\|{\mathbf w}\|_{2} = 1 }\| {\mathbf L}_{N-J,J} {{\mathbf w}} \|_2^2 
$,
and the solution vector ${\mathbf w} \in {\mathbb C}^{J}$ is the right singular vector  corresponding to the smallest singular value of ${\mathbf L}_{N-J,J}$.
Having determined the weight vector ${\mathbf w}$, the rational function $\tilde{r}_{J-1}$ is completely fixed by (\ref{bar-form}).
Finally, we  consider the errors $|\tilde{r}_{J-1}\left(z_{\ell}\right) - g_{\ell}| $ for all $\ell \in \Gamma_{J}$, i.e. for all points $z_\ell$, $\ell \in \Gamma_{J}$, which we do not use for interpolation.
The algorithm terminates if $\max_{\ell \in \Gamma_{J}} | \tilde{r}_{J-1}\left(z_{\ell}\right) - g_{\ell}| < \epsilon$ for a predetermined bound $\epsilon$ or if $J$ reaches a predetermined maximal degree. Otherwise, we find the updated index set 
$$ S_{J+1} \coloneqq S_{J} \cup  \argmax_{\ell \in \Gamma_{J}}| \tilde{r}_{J-1}\left(z_{\ell}\right) - g_{\ell}| $$
and update $\Gamma_{J+1} = I \setminus S_{J+1}$.
\medskip

\begin{algorithm}[ht]\caption{Iterative rational approximation by  AAA algorithm \cite{NST18}} 
\label{alg2} 
\small{
\textbf{Input:} 
$\hat{\mathbf f} \in {\mathbb R}^{N}$ (DCT-II of ${\mathbf f}$)  \\
\phantom{\textbf{Input:}} ${tol}>0$ (tolerance for the approximation error)\\ 
\phantom{\textbf{Input:}} $jmax \in \nn$ with $\textit{jmax} < N/2 $ (maximal order of polynomials in the rational function)

\noindent
\textbf{Initialization:}\\
Set the ordered index set  ${\mathbf \Gamma} \coloneqq \left( k \right)_{k=0}^{N-1}$. Set $z_{k}:=\cos\left( \frac{\pi k}{N}  \right) $ for $k \in {\mathbf \Gamma}$.
Compute ${\mathbf g}_{{\mathbf \Gamma}} = (g_{k})_{k=0}^{N-1}$ by
$$
g_k=(-1)^{k}  \left( \cos\left( \frac{\pi k}{2N}  \right) \right)^{-1} \widehat{f}_k .
$$

\noindent
\textbf{Main Loop:}

\noindent
for $j=1:$ \textit{jmax}
\begin{enumerate}
\item If $j=1$, choose ${\mathbf S}\coloneqq (k)$, $\mathbf{g}_{\mathbf{S}} \coloneqq (g_{k})$, 
where $k \coloneqq \argmax_{\ell \in \mathbf{\Gamma}} |g_{\ell}|$; update $\mathbf{\Gamma}$ and $\mathbf{g}_{{\mathbf \Gamma}}$ by deleting $k$ in $\mathbf{\Gamma}$ and $g_{k}$ in $\mathbf{g}_{{\mathbf \Gamma}}$.\\
 If $ j>1$, compute $k \coloneqq \argmax_{\ell \in \mathbf{\Gamma}} | r_{\ell} - g_{\ell}|$; update $\mathbf{S}$, $\mathbf{g}_{\mathbf{S}}$, $\mathbf{\Gamma}$ and $\mathbf{g}_{{\mathbf \Gamma}}$ by adding $k$ to $\mathbf{S}$ and deleting  $k$ in $\mathbf{\Gamma}$, adding $g_{k}$ to $\mathbf{g}_{\mathbf{S}}$ and deleting it in $\mathbf{g}_{{\mathbf \Gamma}}$.
\item Build  $\mathbf{C}_{N-j,j}\!\coloneqq\!\left( \frac{1}{z_{\ell} - z_{k} } \right)_{\ell \in \mathbf{\Gamma}, k \in \mathbf{S}}$, ${\mathbf L}_{N-j,j}\!\coloneqq\! \left( \frac{g_{\ell} - g_{k}}{z_{\ell} - z_{k}} \right)_{\ell \in \mathbf{\Gamma}, k \in \mathbf{S}}$.
\item Compute the normalized singular vector ${\mathbf w}_{{\mathbf S}}$
 corresponding to the smallest singular value of $\mathbf{L}_{N-j,j}$.
\item Compute $\mathbf{p}\coloneqq\mathbf{C}_{N-j,j} ({\mathbf w}_{{\mathbf S}}.\ast \mathbf{g}_{\mathbf{S}})$, $\mathbf{q} \coloneqq \mathbf{C}_{N-j,j} {\mathbf w}_{{\mathbf S}}$ and $\mathbf{r}=(r_\ell)_{\ell \in \mathbf{\Gamma}}\coloneqq \mathbf{p}./\mathbf{q} \in \rr^{N-j}$, where $.*$ denotes componentwise multiplication and $./$ componentwise division.
\item If $\| \mathbf{r}-\mathbf{g}_{{\mathbf \Gamma}} \|_\infty<$ \textit{tol}  then set $M\coloneqq j-1$ and stop.
\end{enumerate}
end (for)\\
\textbf{Output: } \\ 
$M=j-1$, where $(M,M)$ is the type of the rational function $\tilde{r}_{M}$ \\
$\mathbf{S} \in {\mathbb Z}^{M+1}$ determining the index set $S_{M+1}$ where $\tilde{r}_{M}$ satisfies the interpolation conditions  $\tilde{r}_M(z_k)=g_k$, $k\in S$\\
  $\mathbf{g}_{\mathbf{S}}= (g_{k})_{k \in S_{M+1}} \in \rr^{M+1}$ is the vector of the corresponding interpolation values\\
   ${\mathbf w}_{{\mathbf S}}=(w_k)_{k \in S_{M+1}} \in \rr^{M+1}$ is the weight vector. }

\end{algorithm}

Algorithm \ref{alg2} provides the rational function $\tilde{r}_M(z)$ in a barycentric form 
$\tilde{r}_{M}(z)  = \frac{\tilde{p}_{M}(z)}{\tilde{q}_{M}(z)}$ with 
\begin{equation}\label{bar-form-1}
\tilde{p}_{M}(z) \coloneqq \sum_{k \in S_{M+1}} \frac{w_k \, g_{k}}{z-\cos\left( \frac{\pi k}{N} \right)  }, \qquad  \tilde{q}_{M}(z) \coloneqq \sum_{k \in S_{M+1}} \frac{w_k}{z-\cos\left( \frac{\pi k}{N} \right)  },
\end{equation}
which are determined by the output parameters of this algorithm. Note that it is important to take the occurring index sets and data sets in Algorithm  \ref{alg2}  as ordered sets, therefore they are given as vectors ${\mathbf S}$, ${\mathbf \Gamma}$, ${\mathbf g}_{\mathbf S}$ and ${\mathbf g}_{\mathbf \Gamma}$, as in the original algorithm, \cite{NST18}.

\subsection{Partial fraction decomposition}

In order to rewrite $\tilde{r}_M(z)$ in (\ref{bar-form-1}) in the form  of a partial fraction decomposition,
\begin{equation}\label{rn1} 
r_{M}(z) = \sum_{j=1}^{M} \frac{a_{j}}{z-b_{j}},
\end{equation}
we need to determine $a_1, \ldots , a_M$ and $b_1, \ldots , b_M$ from the output of Algorithm \ref{alg2}. At the same time we force the rational function $\tilde{r}_M(z)$ in (\ref{bar-form-1}) to be of type $(M-1,M)$ of $r_{M}$ in (\ref{rn1}). Note again that in case of exact data $\tilde{r}_M(z)$ in (\ref{bar-form-1}) has indeed type $(M-1,M)$ and coincides with $r_M$ in (\ref{rn1}) (see Corollaries \ref{AAAconv} and \ref{corconv}). 

The zeros of the denominator $\tilde{q}_{M}(z)$ are the poles $b_j$ of $r_M(z)$ and can be computed by solving an $(M+2)\times (M+2)$ generalized eigenvalue problem (see \cite{NST18} or \cite{PPD21}), that has for $S_{M+1}=\{k_{1}, \ldots,  k_{M+1} \}$ the form 
\begin{equation}\label{eig} \left( \begin{array}{ccccc}
0 & w_{k_{1}} & w_{k_{2}} & \ldots & w_{k_{M+1}} \\
1 & \cos\left( \frac{\pi k_1}{N} \right)&   &     & \\
1 & & \cos\left( \frac{\pi k_2}{N} \right)& & \\
\vdots & & & \ddots & \\
1 & & & & \cos\left( \frac{\pi k_{M+1}}{N} \right)\end{array} \right) \, {\mathbf v}_{z}= z \left( \begin{array}{ccccc}
0 & & & & \\
& 1 & & & \\
& & 1 & & \\
& & & \ddots & \\
& & & & 1 \end{array} \right) \, {\mathbf v}_{z}.
\end{equation}
Two eigenvalues of this generalized eigenvalue problem are infinite and the other $M$ eigenvalues are the wanted zeros $b_j$ of $\tilde{q}_{M}(z)$ (see \cite{
Klein,NST18,PPD21} for more detailed explanation). 
We apply Algorithm \ref{alg3} to the output of Algorithm \ref{alg2}.

\begin{algorithm}[ht]\caption{Reconstruction of parameters $a_j$ and $b_j$ of partial fraction representation}
 \label{alg3}
\small{
\textbf{Input:} $\hat{\mathbf f} \in {\mathbb R}^{N}$ DCT-II of ${\mathbf f}$ \\
\phantom{\textbf{Input:}} $\mathbf{S}\in \zz^{M+1}$,
  ${\mathbf{g}}_{\mathbf{S}} \in \rr^{M+1}$,
   ${\mathbf w}_{{\mathbf S}} \in \rr^{M+1}$ the output vectors of Algorithm \ref{alg2}

\begin{enumerate}
\item Build the matrices in (\ref{eig}) and solve this eigenvalue problem to find the  vector ${\mathbf b}^{T}=(b_1, \ldots, b_M)^{T}$ of the $M$ finite eigenvalues;

\item Build the Cauchy matrix $\mathbf{C}_{N,M}=\left(\frac{1}{\cos\big( \frac{\pi k}{N} \big)-b_j} \right)_{k =0,  j=1}^{N-1,M}\in \rr^{N\times M} $ and compute the least squares solution of  the linear system
$$
\mathbf{C}_{N,M} \, \mathbf{a}=  {\mathbf g},
$$
where ${\mathbf g} = (g_{k})_{k=0}^{N-1}$ with $g_{k}:=(-1)^{k}  \left( \cos\left( \frac{\pi k}{2N}  \right) \right)^{-1} \widehat{f}_k $.
\end{enumerate}
\textbf{Output: } Parameter vectors ${\mathbf b}= (b_j)_{j=1}^{M}$, ${\mathbf a} = (a_j)_{j=1}^{M}$ determining $r_M(z)$ in (\ref{rn1}). }
\end{algorithm}

\subsection{Recovery of parameters $\phi_{j}  \in  \frac{\pi}{h N}\zz$}
\label{secper}

For reconstruction of a cosine sum  $f$ in (\ref{1.1}) from exact function values, we still need to study the problem of recovering frequency parameters $\phi_{j}$  satisfying
$\phi_j hN= k \pi$, $k \in \{0, \ldots , N-1\}$, since for these parameters we did not obtain a fractional structure of the DCT-II coefficients $\widehat{f}_{k}$ as exploited in Section \ref{sec31}.
 
Assume now that the function $f$ in (\ref{1.1}) also contains parameters $\phi_{j}$ such that $\phi_j hN= k \pi$, $k \in \{0, \ldots , N-1\}$, and there are exact function values given.  Then we can write $f(t)$ as a sum
$f(t) = f^{(1)}(t)
+ f^{(2)}(t)$, where 
\begin{equation}\label{f1} 
f^{(1)}(t) = \sum_{j=1}^{M_{1}} \gamma_{j} \, \cos(\phi_j t),  \qquad \phi_j h N \not \in \pi \zz, \ j=1,\ldots,M_1,
\end{equation}
and 
\begin{equation}\label{f2}  
f^{(2)}(t) =\sum_{j=M_1+1}^{M} \gamma_{j} \, \cos(\phi_j t),  \qquad \phi_j h N \in \{0, \ldots ,\pi( N-1)\}, \ j=M_1+1,\ldots ,M.
 \end{equation}
Let  $\hat{\mathbf f}^{(1)} = \big((\hat{f}^{(1)})_{k}\big)_{k=0}^{N-1}$  and $\hat{\mathbf f}^{(2)} = \big((\hat{f}^{(2)})_{k}\big)_{k=0}^{N-1}$ 
denote the DCT-II vectors of $\big(f^{(1)}_{\ell}\big)_{\ell=0}^{N-1}$ and $\big(f^{(2)}_{\ell}\big)_{\ell=0}^{N-1}$ respectively. Again we aim at exploiting the special structure of these two DCT-II-vectors.

It follows as in (\ref{dctfor}) that for $k=0,\ldots,N-1$,
$$
(-1)^{k}  \left( \cos\left( \frac{\pi k}{2N}  \right) \right)^{-1} \widehat{f}^{(1)}_k =   \sum\limits_{j=1}^{M_1}  \frac{\gamma_j \sin \left( \frac{ \phi_j h}{2} \right) \sin ( \phi_j h N) }{ \cos\left( \frac{\pi k}{N} \right)- \cos\left( \phi_j h  \right)   }.
 $$
Now we compute $\hat{\mathbf f}^{(2)} = \big(\hat{f}^{(2)}_{\ell}\big)_{\ell=0}^{N-1}$. As in (\ref{fk11}) we have for $k=0, \ldots , N-1$, 
\begin{equation}\label{s2}  
\widehat{f}^{(2)}_{k} = \textstyle \sum\limits_{j=M_{1}+1}^{M} \gamma_{j} \sum\limits_{\ell=0}^{N-1} \frac{1}{2} \left( \cos \left( (2\ell +1) \Big( \frac{\phi_{j} h}{2} + \frac{k\pi}{2N} \Big)\right) + \cos \left( (2\ell +1) \Big( \frac{\phi_{j} h}{2} - \frac{k\pi}{2N} \Big)\right) \right).
\end{equation}
We distinguish three cases. If  $k \not\in \frac{h N}{\pi} \{\phi_{M_{1}+1}, \ldots , \phi_{M}\}$, then $\left(\frac{\phi_{j}h}{2} \pm \frac{k\pi}{2N}\right) \not\in \pi\zz$, 
 and therefore (\ref{3new}) implies
$$ \widehat{f}^{(2)}_{k} = \sum_{j=M_{1}+1}^{M} \frac{\gamma_{j}}{4} \left( \frac{\sin ( \phi_{j} hN +k\pi)}{\sin\left( \frac{\phi_{j}h}{2} + \frac{k\pi}{2N} \right)} + \frac{\sin (\phi_{j} h N -k\pi)}{\sin\left( \frac{\phi_{j}h}{2} - \frac{k\pi}{2N} \right) } \right) =0. $$
If  $k \in \frac{h N}{\pi} \{\phi_{M_{1}+1}, \ldots , \phi_{M}\} \setminus \{ 0 \}$, say $k\pi=\phi_{j'} hN$, then we have  $\left(\frac{\phi_{j}h}{2} \pm \frac{k\pi}{2N} \right) \not\in \pi\zz$  for $j \in \{M_{1}+1, \ldots , M\} \setminus \{j'\}$.  Thus, from (\ref{s2})  applying again  (\ref{3new}) and (\ref{sum1}), we get  
\begin{align}
\widehat{f}^{(2)}_{k}  = & \textstyle \sum\limits_{\substack{j=M_{1}+1 \\ j\neq j'}}^{M} \frac{\gamma_{j}}{4} \left( \frac{\sin ( \phi_{j} hN +k\pi)}{\sin\left( \frac{\phi_{j}h}{2} + \frac{k\pi}{2N} \right)} + \frac{\sin (\phi_{j} h N -k\pi)}{\sin\left( \frac{\phi_{j}h}{2} - \frac{k\pi}{2N} \right) } \right) + \frac{\gamma_{j'}}{2} \sum\limits_{\ell=0}^{N-1} \left( \cos \left( (2\ell+1) \frac{k\pi}{N} \right) + \cos (0) \right)\notag \\
=& \frac{\gamma_{j'} \sin(2k \pi)}{4\sin\left(\frac{k\pi}{N} \right)} +  \textstyle \frac{\gamma_{j'}}{2} \, N= \frac{\gamma_{j'}}{2} \, N. \label{per}
\end{align}
Finally, if there exists a $\phi_{j'} \in \{\phi_{M_{1}+1}, \ldots , \phi_{M} \}$ with $\phi_{j'} =0$, then  (\ref{s2}) implies
$$ \widehat{f}^{(2)}_{0} = \gamma_{j'} N.$$
Summarizing,  we have for $k=0,\ldots,N$
\begin{equation}\label{fk22}
\widehat{f}_{k} = \begin{cases}
\widehat{f}^{(1)}_{k} & k \not\in \frac{h N}{\pi}  \{\phi_{M_{1}+1}, \ldots ,\phi_{M}\}, \\ 
\widehat{f}^{(1)}_{k} + \frac{N}{2} \gamma_{j'} & k=  \frac{\phi_{j'} h N}{\pi}  \in \frac{h N}{\pi}  \{\phi_{M_{1}+1}, \ldots ,\phi_{M}\} \setminus \{0\}, \\
\widehat{f}^{(1)}_{k} + N \gamma_{j'} & k= \frac{\phi_{j'} h N}{\pi} =0 \in \frac{h N}{\pi}  \{\phi_{M_{1}+1}, \ldots ,\phi_{M}\} .
\end{cases}
\end{equation}
In other words, we have $\widehat{f}_{k}= \widehat{f}^{(1)}_{k}$ for $N-M+M_{1}$ indices $k\in \{0, \ldots, N-1\}$.
Similarly as in \cite{DPP21}, we can therefore apply the ESPIRA-I Algorithm \ref{alg1} also in this case.
Then the application of the AAA Algorithm \ref{alg2} in the second step will lead to a rational function $r_{M_{1}}$ of type $(M_{1},M_{1})$ that satisfies $r_{M_{1}}(\cos(\frac{\pi k}{N})) = g_{k} = (-1)^{k} (\cos(\frac{\pi k}{2N}))^{-1} \, \widehat{f}_{k}$ for all $k \not\in \{\phi_{M_{1}+1}, \ldots ,\phi_{M}\}$, while the interpolation values at indices $k \in \{\phi_{M_{1}+1}, \ldots ,\phi_{M}\}$ will be recognized as so-called unachievable points.
Therefore, Algorithm \ref{alg1} will provide all parameters to recover the function $f^{(1)}$.
We refer to Section \ref{secconv}  to show that, for exact input data, Algorithm \ref{alg2}   indeed stops  after $M+1$ iteration steps regardless of occurring integer frequencies.
The periodic function $f^{(2)}$ can be determined in a post-processing step.
Considering the vector $\widehat{\mathbf f }^{(2)} = \left( \widehat{f}^{(2)}_{k} \right)_{k=0}^{N-1} = \left( \widehat{f}_{k} - \widehat{f}^{(1)}_{k}\right)_{k=0}^{N-1}$, we simply recognize the indices  $k$  corresponding to the nonzero components of $\widehat{\mathbf  f}^{(2)}$  and obtain $\phi_{j'} = \frac{k \pi}{h N}$ as well as the corresponding coefficients $\gamma_{j'}$ from (\ref{fk22}).

\begin{remark}
In case of noisy data or for function approximation, this special case of  frequencies $\phi_{j} = \frac{k \pi}{h N} $ with $k \in {\mathbb Z}$ does not usually occur. One indication of frequency parameters close to  $\frac{ \pi}{h N} \zz$ is provided by the weight vector in Algorithm \ref{alg2}. If components of ${\mathbf w}_{\mathbf S}$ are close to zero, then the  corresponding sample value is not interpolated but an unachievable point. In this case, a post-processing step to add frequencies from $\frac{ \pi}{h N} {\mathbb Z}$ may be applied.
\end{remark}

\subsection{Interpolation for exact input data}
\label{secconv}

In this subsection we will study the AAA Algorithm \ref{alg2} in our setting for the reconstruction of cosine sums (\ref{1.1})  from exact input data. We will show that  in this case the AAA algorithm will terminate after $M+1$  iteration steps and provides a rational function of type $(M-1,M)$ that satisfies all interpolation conditions (\ref{intp}). \\
Based on the observations above, we first consider the Loewner matrices ${\mathbf L}_{N-j,j}$ obtained in step 2 of Algorithm \ref{alg2} more closely.
We will prove that ${\mathbf L}_{N-j,j}$ has rank $M$ for any $j$ with $M\le j \le N-M$. 

\begin{theorem}\label{theo2}
Let $f$ be an $M$-sparse cosine sum as in $(\ref{1.1})$, and let ${\mathbf f} =( f_{\ell})_{\ell=0}^{N-1}$ with $f_{\ell} = f\left(\frac{h(2\ell+1)}{2} \right)$ be given, where $N>2M$. Further let $\hat{\mathbf f} = {\mathbf C}_{N}^{{II}} \, {\mathbf f}$  be the DCT-II transformed vector with ${\mathbf C}_{N}^{{II}}$ as in $(\ref{CC})$. Then, for any partition $S \cup \Gamma$ of $\{0, \ldots , N-1\}$ where both subsets have at least $M$ elements, i.e., $|S| \ge M$ and $|\Gamma| \ge M$, it follows that the Loewner matrix
$$ \textstyle {\mathbf L}_{|\Gamma|, |S|} = \left( \frac{g_{\ell} - g_{k}}{\cos \big(\frac{\pi \ell}{N}\big)- \cos \big(\frac{\pi k}{N} \big)} \right)_{\ell \in \Gamma, k \in S}  $$
with $g_{k} = (-1)^{k} \big(\cos \big(\frac{\pi k}{2N}\big)\big)^{-1} \hat{f}_{k}$ for $k=0, \ldots , N-1$, has exactly rank $M$.
\end{theorem}

\begin{proof}
1. Assume first that all frequencies $\phi_{j}$ satisfy $\phi_{j} \not\in \frac{ \pi}{h N} {\mathbb Z}$. 
Then by (\ref{dctfor}) we obtain 
\begin{align}
{\mathbf L}_{|\Gamma|, |S|} &= \left(\frac{\sum\limits_{j=1}^{M}  \frac{\gamma_j \, \sin \big(\frac{\phi_j h}{2} \big) \sin ( \phi_j hN) }{ \cos \big(\frac{\pi \ell}{N} \big)- \cos \big(\phi_j h \big)    } -  \sum\limits_{j=1}^{M}  \frac{\gamma_j \, \sin \big( \frac{\phi_j h}{2} \big) \sin (\phi_j hN) }{ \cos \big(\frac{\pi k}{N} \big)- \cos \big(\phi_j h\big)     }}{\cos \big(\frac{\pi \ell}{N} \big)- \cos \big(\frac{\pi k}{N} \big)} \right)_{\ell \in \Gamma, k \in S} \notag\\
&= \left(\frac{\sum\limits_{j=1}^{M} \gamma_j \, \sin  \big(\frac{\phi_j h}{2} \big) \sin ( \phi_j hN)  \frac{ \left(\cos \big(\frac{\pi k}{N} \big)- \cos \big(\phi_j h\big)   - \cos \big(\frac{\pi \ell}{N}\big) + \cos \big(\phi_j h\big) \right)}{\left(\cos \big( \frac{\pi k}{N} \big)- \cos \big(\phi_j h\big) \right) \left(\cos \big(\frac{\pi \ell}{N} \big)- \cos \big(\phi_j h\big)  \right)}}{\cos \big(\frac{\pi \ell}{N} \big) - \cos \big(\frac{\pi k}{N}\big)} \right)_{\ell \in \Gamma, k \in S} \notag\\
&= \left( - \sum_{j=1}^{M} \frac{\gamma_{j} \, \sin \big( \frac{\phi_{j} h}{2} \big) \sin ( \phi_{j} hN )}{\left(\cos \big(\frac{\pi k}{N} \big)- \cos \big( \phi_j h\big) \right) \left(\cos \big(\frac{\pi \ell}{N}\big) - \cos \big(\phi_j h\big)  \right)} \right)_{\ell \in \Gamma, k \in S} \notag\\
&=  \textstyle {\mathbf C}_{|\Gamma|,M}  \, \textrm{diag} \left( -\gamma_{j} \, \sin \big( \frac{\phi_{j} h}{2} \big) \sin \big(  \phi_{j} hN\big) \right)_{j=1}^{M} \, {\mathbf C}_{|S|,M}^{T} \label{factL}
\end{align}
with the Cauchy matrices
$$ \textstyle {\mathbf C}_{|\Gamma|,M} := \left( \frac{1}{\cos \big(\frac{\pi \ell}{N}\big)- \cos\big(\phi_j h \big)} \right)_{\ell \in \Gamma, j=1, \ldots, M}, \quad 
{\mathbf C}_{|S|,M}:= \left( \frac{1}{\cos \big( \frac{\pi k}{N}\big) - \cos \big(\phi_j h\big)} \right)_{k \in S, j=1, \ldots, M}. $$
The assertion now directly follows from this factorization since ${\mathbf C}_{|\Gamma|,M}$ and ${\mathbf C}_{|S|,M}$ have full column rank $M$, while the diagonal matrix has full rank by assumption.

2. Assume now that $f$ in (\ref{1.1})  also contains  frequencies $\phi_{j} \in \frac{ \pi}{h N} {\mathbb Z}$.
Then we can apply our considerations from  Section \ref{secper}. Assume that $f= f^{(1)} + f^{(2)}$ where $f^{(1)}$ contains the frequencies $\phi_{j} \not\in \frac{ \pi}{h N}{\mathbb Z}$ for $j=1, \ldots , M_{1}$, and $f^{(2)}$ contains the frequencies $\phi_{j}= \frac{ \pi}{h N}k_{j}$ with $k_{j} \in \{0, \ldots , N-1\}$ for $j=M_{1}+1, \ldots, M$.
We denote $g_{k}^{(1)} := (-1)^{k} \big(\cos \big(\frac{\pi k}{2N}\big)\big)^{-1} \hat{f}^{(1)}_{k}$ and 
$g_{k}^{(2)} := (-1)^{k} \big(\cos\big(\frac{\pi k}{2N}\big)\big)^{-1} \hat{f}^{(2)}_{k}$.
Then,  by (\ref{dctfor}) and (\ref{fk22}) 
we still have for all $k$ with $k \notin  \frac{ hN}{\pi} \{ \phi_{M_1+1}, \ldots , \phi_{M}\}$,
\begin{align*}
g_{k} &= \textstyle g_{k}^{(1)} = (-1)^{k} \big(\cos \big(\frac{\pi k}{2N} \big) \big)^{-1} \hat{f}^{(1)}_{k}  
= \sum\limits_{j=1}^{M_{1}}  \frac{\gamma_j \, \sin \big( \frac{h \phi_j }{2} \big) \sin (\phi_j hN) }{ \cos\big( \frac{\pi k}{N} \big)- \cos\big( \phi_j h\big)},\\
g_{k}^{(2)} &= 0,
\end{align*}
while for all $k_{j} = \frac{\phi_{j} hN}{\pi}$, $j=M_{1}+1, \ldots , M$, 
\begin{align*}
 g_{k_{j}}  &  = \textstyle (-1)^{k_{j}} \big(\cos \big(\frac{\pi k_{j}}{2N}\big)\big)^{-1} \big(\hat{f}^{(1)}_{k_{j}} + \hat{f}^{(2)}_{k_{j}} \big)=g_{k_{j}}^{(1)} + g_{k_{j}}^{(2)}\\
 &= \textstyle (-1)^{k_{j}} \big(\cos \big(\frac{\pi k_{j}}{2N}\big)\big)^{-1}  \hat{f}^{(1)}_{k_{j}} + (-1)^{k_{j}} \big(\cos \big(\frac{\pi k_{j}}{2N} \big) \big)^{-1} \epsilon_{j} N \frac{\gamma_{j}}{2}
 \end{align*}
with $\epsilon_{j} = 1$ for $k_{j} \neq 0$ and $\epsilon_{j} = 2$ for $k_{j} = 0$.
Assume that  $ k_{j} = \frac{\phi_{j} hN}{\pi} \in \Gamma$ for $j=M_{1}+1, \ldots , M$. 
Then it follows that the Loewner matrix ${\mathbf L}_{|\Gamma|,|S|}$ is of the form 
\begin{align}
 {\mathbf L}_{|\Gamma|,|S|} &= \widetilde{\mathbf L}_{|\Gamma|,|S|}+ \overset{\approx}{\mathbf L}_{|\Gamma|,|S|}, \label{loew}
 \end{align}
where  
\begin{align}
\widetilde{\mathbf L}_{|\Gamma|,|S|} &:=  \left( \frac{g_{\ell}^{(1)} - g_{k}^{(1)}}{\cos \big(\frac{\pi \ell}{N} \big)- \cos \big(\frac{\pi k}{N}\big)} \right)_{\ell \in \Gamma, k \in S} \notag \\
&=\left(\frac{\sum\limits_{j=1}^{M_{1}}  \frac{\gamma_j \, \sin \big( \frac{\phi_j h}{2} \big) \sin (\phi_j hN) }{ \cos\big( \frac{\pi \ell}{N} \big)- \cos\big( \phi_j h \big)   } -  \sum\limits_{j=1}^{M_{1}}  \frac{\gamma_j \, \sin \big( \frac{\phi_j h}{2} \big) \sin (\phi_j hN) }{ \cos\big( \frac{\pi k}{N} \big)- \cos\big( \phi_j h  \big)   }}{\cos \big(\frac{\pi \ell}{N} \big)- \cos\big(\frac{\pi k}{N}\big)} \right)_{\ell \in \Gamma, k \in S} \notag \\
&=  {\mathbf C}_{|\Gamma|,M_{1}}  \, \textrm{diag} \left( - \gamma_{j} \, \sin \big( \frac{\phi_{j} h}{2} \big) \sin \left( \phi_{j} hN \right) \right)_{j=1}^{M_{1}} \, {\mathbf C}_{|S|,M_{1}}^{T} \label{loew1}
\end{align}
corresponds to the function $f^{(1)}$ 
and where
\begin{align} 
\overset{\approx}{\mathbf L}_{|\Gamma|,|S|} &: = \textstyle \left( \frac{  g_{\ell}^{(2)} - g_{k}^{(2)}}{\cos \big(\frac{\pi \ell}{N}\big) - \cos\big(\frac{\pi k}{N}\big)} \right)_{\ell \in \Gamma, k \in S} =  \left( \frac{  g_{\ell}^{(2)}}{\cos \big(\frac{\pi \ell}{N}\big) - \cos\big(\frac{\pi k}{N}\big)} \right)_{\ell \in \Gamma, k \in S} \notag \\
&= \textstyle 
\sum\limits_{j=M_{1}+1}^{M} \epsilon_{j} \, N \frac{\gamma_{j}}{2} (-1)^{k_{j}}\big(\cos \big(\frac{\pi k_{j}}{2N}\big)\big)^{-1}\, {\mathbf e}_{\mathrm{ind}(k_{j})} \ \left( \frac{1}{\cos\big(\frac{\pi \, k_{j} }{N}\big) - \cos\big(\frac{\pi k}{N}\big) } \right)^{T}_{k \in S} \label{loew2}
\end{align}
corresponds to $f^{(2)}$. 
Here ${\mathbf e}_{\mathrm{ind}(k_{j})} $ denotes the $\mathrm{ind}(k_{j})$-th unit vector of length $|\Gamma|$ and $\mathrm{ind}(k_{j})$ denotes the position, or index,
of $k_j$ in the ordered set $\Gamma$.
Therefore we have rank $\widetilde{\mathbf L}_{|\Gamma|,|S|} = M_{1}$ and rank $\overset{\approx}{\mathbf L}_{|\Gamma|,|S|} = M-M_{1}$, such that rank ${\mathbf L}_{|\Gamma|,|S|} \le \mathrm{rank} \,\widetilde{\mathbf L}_{|\Gamma|,|S|} + \mathrm{rank} \,\overset{\approx}{\mathbf L}_{|\Gamma|,|S|} = M$.
The image of $\widetilde{\mathbf L}_{|\Gamma|,|S|}$ is spanned by the $M_{1}$ independent columns of the Cauchy matrix 
${\mathbf C}_{|\Gamma|,M_{1}} = \Big( \frac{1}{\cos \big(\frac{\pi \ell}{N}\big) - \cos \big(\phi_j h\big)} \Big)_{\ell \in \Gamma, j=1}^{M_{1}}$, while the image of $\overset{\approx}{\mathbf L}_{|\Gamma|,|S|}$ is spanned by the $M-M_{1}$ unit vectors ${\mathbf e}_{\mathrm{ind}(k_{j})}$, $j=M_{1}+1, \ldots , M$. Since these two spans are linearly independent, we have indeed  rank ${\mathbf L}_{|\Gamma|,|S|}= M$.
The remaining cases, where either  $k_{j} = \frac{\phi_{j} hN}{\pi}$, $j=M_{1}+1, \ldots , M$,  are all contained in $S$, or that the set $\{k_{j} = \frac{\phi_{j} hN}{\pi}, \, j=M_{1}, \ldots , M\}$ is split into two subsets, one contained in $\Gamma$ and one in $S$, can be treated similarly.
\end{proof}

Using Theorem \ref{theo2} we further show that  Algorithm \ref{alg2} will stop after  step $M+1$, since we will find a singular vector ${\mathbf w}$ in the kernel of ${\mathbf L}_{N-M-1,M+1}$. 
Thus, the minimal value of the sum in (\ref{mini}) will be equal to zero. First we consider the case when all parameters $\phi_{j}$ are not in $\frac{ \pi}{h N} {\mathbb Z}$.

\begin{corollary}\label{AAAconv}
Let $f$ be an $M$-sparse cosine sum as in $(\ref{1.1})$ with $\phi_{j} \not\in \frac{ \pi}{h N} {\mathbb Z}$ for $j=1,\ldots,M$, and let ${\mathbf f} =( f_{\ell})_{\ell=0}^{N-1}$ with $f_{\ell} = f\left(\frac{h(2\ell+1)}{2} \right)$ be given, where $N>2M$. Further let $\hat{\mathbf f} = {\mathbf C}_{N}^{{II}} \, {\mathbf f}$  be the DCT-II transformed vector with ${\mathbf C}_{N}^{{II}}$ as in $(\ref{CC})$.  Then Algorithm $\ref{alg2}$  terminates after  $M+1$ steps and determines 
 a partition $S_{M+1} \cup \Gamma_{M+1}$ of $I\coloneqq \{0, \ldots , N-1\}$ with $|S_{M+1}|=M+1$ and $|\Gamma_{M+1}| = N-M-1$ and a rational function $r_{M}(z)$ of type ($M-1, M)$ satisfying  interpolation conditions 
 \begin{equation}\label{intthe}
\textstyle r_M \left( \cos\left( \frac{\pi k}{N} \right)  \right) =(-1)^{k}
 \left( \cos\left( \frac{\pi k}{2N}  \right) \right)^{-1} \widehat{f}_k
\end{equation}
for $k=0,\ldots,N-1$.
 \end{corollary}
\begin{proof}
First observe that a rational function $r_{M}(z) = \frac{p_{M-1}(z)}{q_{M}(z)} $ with numerator polynomial  $p_{M-1}(z)$ of degree at most $M-1$ and denominator polynomial $q_{M}(z)$ of degree $M$ is already completely  determined by $2M$ (independent) interpolation conditions 
 if the rational interpolation problem is solvable at all.  But solvability is clear because of (\ref{dctfor}).  In particular, the given data $(-1)^{k}
 \left( \cos\big( \frac{\pi k}{2N}  \big) \right)^{-1} \widehat{f}_k$, $k \in I$, cannot be interpolated by a rational function of smaller type than $(M-1,M)$.

Assume now that at the $(M+1)$-st iteration step in Algorithm \ref{alg2}, the index set $S_{M+1} \subset I$ with $M+1$ interpolation indices has been chosen, and let $\Gamma_{M+1} = I \setminus S_{M+1}$.
Then the Loewner matrix ${\mathbf L}_{N-M-1,M+1}$ obtained after $M+1$ iterations steps has rank $M$ according to Theorem \ref{theo2}. 
Therefore, the kernel of ${\mathbf L}_{N-M-1,M+1}$ has dimension $1$, and the normalized  vector ${\mathbf w}={\mathbf w}_{S}$ in step 3 of Algorithm \ref{alg2} satisfies ${\mathbf L}_{N-M-1,M+1} \, {\mathbf w}={\mathbf 0}$.
According to (\ref{factL}) we observe that 
$$
{\mathbf C}_{M+1,M}^{T} \, {\mathbf w} =\left( \frac{1}{\cos \big( \frac{\pi k}{N}\big) - \cos \big(\phi_j h\big)} \right)_{ j=1, \ldots, M, \ k \in S_{M+1}} \, {\mathbf w} = {\mathbf 0}. 
 $$
Since any $M$ columns of the Cauchy matrix ${\mathbf C}_{M+1,M}^{T}$ are linearly independent, we conclude that all components of ${\mathbf w}$ are nonzero.
By construction it follows that a rational function $\tilde{r}_M$ in (\ref{bar-form-1}) satisfies  interpolation conditions (\ref{intthe}) for $k \in S_{M+1}$. On the other hand, since ${\mathbf L}_{N-M-1,M+1} \, {\mathbf w}={\mathbf 0}$ we find that
$$
\sum\limits_{k\in S_{M+1}}\left(\frac{g_{\ell} w_k}{\cos\frac{\pi \ell}{N} -\cos\frac{\pi k}{N} } - \frac{g_{k} w_k}{\cos\frac{\pi \ell}{N} -\cos\frac{\pi k}{N} }   \right)=0,  \  \  \ell\in \Gamma_{M+1}, 
$$
with $g_k=(-1)^{k}
 \left( \cos\big( \frac{\pi k}{2N}  \big) \right)^{-1} \widehat{f}_k$,
which implies that for $\tilde{r}_M$ in (\ref{bar-form-1}) the interpolation conditions (\ref{intthe}) hold also for $\ell \in \Gamma_{M+1}$. Since a rational function satisfying all interpolation conditions is uniquely determined, the obtained rational function $\tilde{r}_M$ in barycentric form (\ref{bar-form-1}) coincides with $r_M$ in (\ref{rat}) and needs to have type $(M-1,M)$. That means that Algorithm \ref{alg2} provides the wanted rational function $r_{M}(z)$. 
\end{proof}

\noindent
Similarly, we show the following result for the case of 
 frequency parameters $\phi_{j} \in \frac{ \pi}{h N} {\mathbb Z}$.
\begin{corollary}\label{corconv}
Let $f(t) = f_{1}(t)+ f_{2}(t)$ be of the form $(\ref{f1})$ and $(\ref{f2})$  with $0 \le M_{1} <M \in {\mathbb N}$.
Let ${\mathbf f} =( f_{\ell})_{\ell=0}^{N-1}$ with $f_{\ell} = f\left(\frac{h(2\ell+1)}{2} \right)$ be given, where $N>2M$. Further let $\hat{\mathbf f} = {\mathbf C}_{N}^{{II}} \, {\mathbf f}$  be the DCT-II transformed vector with ${\mathbf C}_{N}^{{II}}$ as in $(\ref{CC})$.   Then the modification of Algorithm $\ref{alg2}$  described in Section $\ref{secper}$ terminates after  $M+1$ steps and determines 
 a partition $S_{M+1} \cup \Gamma_{M+1}$ of $I \coloneqq \{0, \ldots , N-1\}$ and a rational function $r_{M_{1}}(z)$ of type ($M_{1}-1, M_{1})$ satisfying  
 \begin{equation}\label{intthe1}
 \textstyle r_{M_1} \left( \cos\big( \frac{\pi k}{N} \big)  \right) =(-1)^{k}
 \left( \cos\big( \frac{\pi k}{2N}  \big) \right)^{-1} \widehat{f}^{(1)}_k,
\end{equation}
 $k=0, \ldots , N-1$.
\end{corollary}
\begin{proof}
At the $(M+1)$-st iteration step of Algorithm \ref{alg2} we obtain a Loewner matrix \linebreak
${\mathbf L}_{N-M-1, M+1}$, which according to Theorem \ref{theo2} has rank $M$. Similarly  as in the proof of Theorem \ref{theo2} we can show that a Loewner matrix  ${\mathbf L}_{N-J, J}$, computed at the $J$-th iteration step of Algorithm \ref{alg2} with $J<M+1$, has full column rank $J$. Thus the Algorithm \ref{alg2} does not stop earlier than after $M+1$ iteration steps.

At the $(M+1)$-st iteration step of Algorithm \ref{alg2}  we obtain a partition $S_{M+1}\cup \Gamma_{M+1}$ of the set $I$. 
Similarly as in the proof of Theorem \ref{theo2}, we assume that $k_{j} = \frac{\phi_{j} hN}{\pi} \in \Gamma_{M+1}$ for $j=M_{1}+1, \ldots , M$. Then the Loewner matrix ${\mathbf L}_{N-M-1, M+1}$ has the structure as in (\ref{loew}) with sets $\Gamma_{M+1}$ and $S_{M+1}$ instead of $\Gamma$ and $S$ respectively.  Since the Loewner matrix ${\mathbf L}_{N-M-1, M+1}$ has rank $M$, the Loewner matrix $\tilde{\mathbf L}_{N-M-1, M+1}$ has rank $M_1$ due to factorization (\ref{loew1}) and we have $M-M_1$ frequency parameters $\phi_{j}$ satisfying  $\phi_j h N \in \{0, \ldots ,\pi( N-1)\}, \ j=M_1+1,\ldots ,M$.
Then, each of $M-M_1$ rank-1 matrices in (\ref{loew2}) enlarges the rank of ${\mathbf L}_{N-M-1, M+1}$ by 1. 
Therefore, the normalized kernel vector  ${\mathbf w} = {\mathbf w}_{S}\in \cc^{M+1}$ in step 3 of Algorithm \ref{alg2} is uniquely defined and satisfies  ${\mathbf L}_{N-M-1,M+1} \, {\mathbf w}={\mathbf 0}$ as well as $\tilde{{\mathbf L}}_{N-M-1,M+1} \, {\mathbf w}={\mathbf 0}$ and $\left( \frac{1}{\cos\big(\frac{\pi \, k_{j} }{N}\big) - \cos\big(\frac{\pi k}{N}\big) } \right)^{T}_{k \in S_{M+1}}  \, {\mathbf w}={\mathbf 0}$.
Let $\tilde{r}_M$ be a rational function in barycentric form (\ref{bar-form-1}) constructed by Algorithm \ref{alg2} after $M+1$ iteration steps. Then from $\tilde{{\mathbf L}}_{N-M-1,M+1} \, {\mathbf w}={\mathbf 0}$ we get the interpolation conditions (\ref{intthe1}) for $k \in \Gamma_{M+1}$. According to the construction procedure we find  (\ref{intthe1})  also for $k\in S_{M+1}$.
Further we simplify $\tilde{r}_{M}$ by removing all zero components of ${\mathbf w}$ in definition (\ref{bar-form-1}). Then, instead of  $\tilde{r}_{M}$ we obtain a rational function $r_{M_1}$ of type $(M_1-1,M_1)$ satisfying (\ref{intthe1}). 
If the frequencies $k_{j} = \frac{\phi_{j} hN}{\pi}$ are contained in $S_{M+1}$ or in both index sets $S_{M+1}$ and  $\Gamma_{M+1}$, the assertions can be shown similarly.
\end{proof}

\subsection{ESPIRA II: ESPIRA as matrix pencil method for Loewner matrices}

At the first glance, the two algorithms ESPRIT in Algorithm \ref{algESPRIT} and ESPIRA-I in Algorithm \ref{alg1} for reconstruction of cosine sums seem to be completely unrelated.
In this section, we will show that the ESPIRA algorithm can be also understood as a matrix pencil method, but for Loewner instead of Hankel$+$Toeplitz matrices.

Next, we want to show how the wanted frequency parameters  $\phi_{j}$ of $f$ in (\ref{1.1}) can be also obtained by solving a matrix pencil problem for Loewner matrices.
\begin{theorem}\label{theo3}
Let $f$ be an $M$-sparse cosine sum as in $(\ref{1.1})$, and let ${\mathbf f} =( f_{\ell})_{\ell=0}^{N-1}$ with $f_{\ell} = f\big(\frac{h(2\ell+1)}{2} \big)$ be given, where $N>2M$. Further let $\hat{\mathbf f} = {\mathbf C}_{N}^{{II}} \, {\mathbf f}$  be the DCT-II transformed vector with ${\mathbf C}_{N}^{{II}}$ as in $(\ref{CC})$. 
Let  $S_{M} \cup \Gamma_{M}$ be a partition of $\{0, \ldots , N-1\}$ with $|S_{M}|= M$ and $|\Gamma_{M}| = N-M$. Then the values $z_{j}= \cos \big(\phi_{j} h\big)$, $j=1, \ldots, M$, with frequencies $\phi_{j}$ of $f$ in $(\ref{1.1})$
are the eigenvalues of the Loewner matrix pencil 
\begin{equation}\label{mpl} z {\mathbf L}^{(0)}_{N-M,M} - {\mathbf L}^{(1)}_{N-M,M}
\end{equation}
with 
$$ \textstyle 
{\mathbf L}^{(0)}_{N-M,M}= \left( \frac{g_{\ell} - g_{k}}{\cos \big(\frac{\pi \ell}{N}\big) - \cos \big(\frac{\pi k}{N}\big)} \right)_{\ell \in \Gamma_{M}, k \in S_{M}},
 {\mathbf L}^{(1)}_{N-M,M}= \left( \frac{g_{\ell} \, \cos \big(\frac{\pi \ell}{N}\big)- g_{k} \, \cos \big(\frac{\pi k}{N}\big)}{\cos \big( \frac{\pi \ell}{N}\big) - \cos \big(\frac{\pi k}{N}\big)} \right)_{\ell \in \Gamma_{M}, k \in S_{M}}, $$
where $g_{k} = (-1)^{k} \big(\cos\big(\frac{\pi k}{2N}\big)\big)^{-1} \hat{f}_{k}$ for $k=0, \ldots , N-1$.
\end{theorem}

\begin{proof}
1. Assume first that all wanted frequencies $\phi_{j}$ are not in $\frac{\pi}{hN} {\mathbb Z}$. Then, as in (\ref{factL}),  we find for ${\mathbf L}^{(0)}_{N-M,M}$ the factorization 
\begin{align*} \textstyle 
{\mathbf L}^{(0)}_{N-M,M} = {\mathbf C}_{N-M,M}  \, \textrm{diag} \left(- \gamma_{j} \, \sin \big( \frac{\phi_{j} h}{2} \big) \sin \left(  \phi_{j} hN \right) \right)_{j=1}^{M} \, {\mathbf C}_{M,M}^{T}
\end{align*}
with Cauchy matrices 
$$ \textstyle {\mathbf C}_{N-M,M} := \left( \frac{1}{\cos \big(\frac{\pi \ell}{N}\big) - \cos ( \phi_j h) }\right)_{\ell \in \Gamma_{M}, j=1, \ldots , M}, \;  
{\mathbf C}_{M,M}:= \left( \frac{1}{\cos \big(\frac{\pi k}{N}\big)- \cos( \phi_jh) } \right)_{k \in S_{M}, j=1,  \ldots , M}.$$ 
For the second Loewner matrix we find with (\ref{dctfor})
\begin{align*}
& {\mathbf L}^{(1)}_{N-M,M} 
= \textstyle
\left(   \frac{ \cos \big(\frac{\pi\ell}{N}\big) \, \sum\limits_{j=1}^{M}  \frac{\gamma_j \, \sin \big( \frac{\phi_j h}{2} \big) \sin (\phi_j hN) }{ \cos\big(\frac{\pi \ell}{N}\big) - \cos (\phi_j h)     } -  \cos\big(\frac{\pi k}{N}\big)\,  \sum\limits_{j=1}^{M}  \frac{\gamma_j \, \sin \big( \frac{\phi_j h}{2} \big) \sin (\phi_j hN) }{ \cos\big(\frac{\pi k}{N}\big) - \cos (\phi_j h)    }}{\cos \big(\frac{\pi \ell}{N}\big) - \cos \big(\frac{\pi k}{N}\big)} \right)_{\ell \in \Gamma_{M}, k \in S_{M}}\\
&= \!\textstyle \left( \sum\limits_{j=1}^{M}\!  \frac{\gamma_j \, \sin \! \big(\frac{\phi_j h}{2}\big)   \sin (\phi_j hN)  \left(\cos \!\big(\frac{\pi\ell}{N}\big)  \big(\cos \! \big(\frac{\pi k}{N}\big) - \cos(\phi_{j} h)\big) - \cos\big(\frac{\pi k}{N}\big) \big(\cos \big(\frac{\pi \ell}{N}\big) - \cos(\phi_{j}h)\big)\right)}{\big(\cos \left(\frac{\pi k}{N}\right) - \cos(\phi_{j} h)\big) \big(\cos\left(\frac{\pi \ell}{N}\right) - \cos(\phi_{j} h)\big) \big(\cos \left(\frac{\pi \ell}{N}\right) - \cos \frac{\pi k}{N}\big)} \!\right)_{\!\ell \in \Gamma_{M}, k \in S_{M}} \\
&= \textstyle -\left( \sum\limits_{j=1}^{M}  \frac{\gamma_j \, \sin  \big(\frac{\phi_j h}{2} \big)\,  \sin (\phi_j hN) \, \cos \big(\phi_{j}h \big)}{\big(\cos \big(\frac{\pi k}{N}\big) - \cos \big(\phi_{j} h\big) \big) \big(\cos \big(\frac{\pi \ell}{N}\big) - \cos\big(\phi_{j} h\big)\big) } \right)_{\ell \in \Gamma, k \in S}\\
&= \textstyle {\mathbf C}_{N-M,M}  \, \textrm{diag} \left( -\gamma_{j} \, \sin \big( \frac{\phi_{j} h}{2} \big)\, \sin (\phi_{j} hN )\,  \cos \big(\phi_{j} h\big) \!\right)_{j=1}^{M} \, {\mathbf C}_{M,M}^{T}.
\end{align*}
Therefore
\begin{align*} \textstyle
& z {\mathbf L}^{(0)}_{N-M,M} - {\mathbf L}^{(1)}_{N-M,M}\\
& =  \textstyle {\mathbf C}_{N-M,M}  \, \textrm{diag} \left( -\gamma_{j} \, \sin  \big(\frac{\phi_{j} h}{2}\big) \, \sin (\phi_{j} hN) \,  \big(z-\cos(\phi_{j}h)\big) \right)_{j=1}^{M} \, {\mathbf C}_{M,M}^{T}
\end{align*}
has a reduced rank $M-1$ if and only if $z-\cos(\phi_{j}h) = 0$. 

2. Assume now that $\phi_{j} \not\in \frac{\pi}{hN}{\mathbb Z}$ for $j=1,\ldots , M_{1}$ and that $\frac{h N}{\pi} \phi_{j} = k_{j} \in \{0, \ldots , N-1\}$ for $j=M_{1}+1, \ldots , M$. Similarly as in the proof of Theorem \ref{theo2}, assume for simplicity that all these indices $k_{j}$ are in $\Gamma_{M}$. Then we find as in the proof of Theorem \ref{theo2} that 
\begin{align*}
{\mathbf L}_{N-M,M}^{(0)} &=  \textstyle
{\mathbf C}_{N-M,M_{1}}  \, \textrm{diag} \left( -\gamma_{j} \, \sin \big( \frac{\phi_{j} h}{2} \big) \sin \big( \phi_{j} hN\big) \right)_{j=1}^{M_{1}} \, {\mathbf C}_{M,M_{1}}^{T} \\
& + \textstyle
 \sum\limits_{j=M_{1}+1}^{M} \epsilon_{j} N\, (-1)^{k_{j}} \, \big(\cos \big(\frac{k_{j} \pi}{2N}\big)\big)^{-1} \, \frac{\gamma_{j}}{2} \, {\mathbf e}_{\mathrm{ind}(k_{j})} \ \left( \frac{1}{\cos\big(\frac{ \pi \,  k_{j} }{N}\big) - \cos\big(\frac{\pi k}{N}\big) } \right)^{T}_{k \in S_{M}}. 
\end{align*}
For the second matrix we find analogously as above
\begin{align*}
{\mathbf L}_{N-M,M}^{(1)} & =  \textstyle
{\mathbf C}_{N-M,M_{1}}  \, \textrm{diag} \left(- \gamma_{j} \, \sin \big(\frac{\phi_{j} h}{2} \big)\, \sin  ( \phi_{j} hN) \, \cos \big(\phi_{j} h \big) \right)_{j=1}^{M_{1}} \, {\mathbf C}_{M,M_{1}}^{T} \\
& + \textstyle
 \sum\limits_{j=M_{1}+1}^{M} \epsilon_{j} N\, (-1)^{k_{j}} \, \big(\cos \big(\frac{k_{j} \pi}{2N}\big)\big)^{-1} \, \frac{\gamma_{j}}{2} \, {\mathbf e}_{\mathrm{ind}(k_{j})}  \ \left( \frac{\cos \big( \phi_{j} h\big)}{\cos\big(\frac{\pi\,  k_{j}}{N}\big) - \cos\big(\frac{\pi k}{N}\big) } \right)^{T}_{k \in S_{M}}, 
\end{align*}
where we have used the fact that $\cos \big(\frac{\pi k_{j}}{N} \big)= \cos \big(\phi_{j} h \big)$. 
Thus again, the matrix $z {\mathbf L}_{N-M,M}^{(0)}- {\mathbf L}_{N-M,M}^{(1)}$ has only rank $M-1$ if $z= \cos \big(\phi_{j} h\big)$ for $j \in \{ 1, \ldots , M\}$, where either the first matrix difference 
$$\textstyle {\mathbf C}_{N-M,M_{1}}  \, \textrm{diag} \left( -\gamma_{j} \, \sin \big(\frac{\phi_{j} h}{2} \big)\, \sin ( \phi_{j} hN )\, \big(z-\cos \big( \phi_{j} h \big)\big) \right)_{j=1}^{M_{1}} \, {\mathbf C}_{M,M_{1}}^{T}$$ 
has only rank $M_{1}-1$ (if $j \in \{1, \ldots , M_{1}\}$), or the second matrix difference 
$$ \textstyle  \sum\limits_{j=M_{1}+1}^{M} \epsilon_{j} N\, (-1)^{k_{j}} \, \big(\cos \big(\frac{k_{j} \pi}{2N}\big)\big)^{-1} \, \frac{\gamma_{j}}{2} \, {\mathbf e}_{\mathrm{ind}(k_{j})}  \ \left( \frac{ z-\cos \big(\phi_{j} h\big)}{\cos\big(\frac{\pi \,   k_{j}  }{N}\big) - \cos\big(\frac{\pi k}{N}\big) } \right)^{T}_{k \in S_{M}}$$ 
has only rank $M-M_{1}-1$ instead of $M-M_{1}$ because of one vanishing term (if $z= \cos \big(\phi_{j}h\big)$ for $j \in \{M_{1}+1, \ldots, M\})$.
All further cases can be treated similarly.
\end{proof}

Theorem \ref{theo3} provides now a new approach to reconstruct all frequencies $\phi_{j}$ of $f$ in (\ref{1.1}).
If the number of terms $M$ is known, then we can theoretically use any arbitrary partition $S_{M} \cup \Gamma_{M}$ to build the two Loewner matrices in Theorem \ref{theo3} and to extract the frequencies $\phi_{j} \in [0,K]$ by solving the matrix pencil problem (\ref{mpl}). 

However in order to obtain better numerical stability, we will apply the greedy procedure of the AAA Algorithm \ref{alg2} to build the index sets $S_{M}$ and $\Gamma_{M}$ and use this partition to construct the Loewner matrices.
Moreover, incorporating this preconditioning procedure, we can still determine $M$. The complete algorithm for ESPIRA-II based on the matrix pencil method for Loewner matrices is summarized in Algorithm \ref{alg5}.

\begin{algorithm}[tp]\caption{ESPIRA II}
\label{alg5}
\small{
\textbf{Input:} ${\mathbf f} = (f_{\ell})_{\ell=0}^{N-1} = \big(f\left(\frac{h(2\ell+1)}{2}\right)\big)_{\ell=0}^{N-1}$ (equidistant sampling values of $f$ in (\ref{1.1}))\\
for exact sampling data (reconstruction): $jmax =\lfloor N/2 \rfloor-1$ (upper bound for M), ${tol}>0$ tolerance for the approximation error\\
for noisy sampling data (approximation): $jmax= M+1$ ($M$ wanted length of cosine sum).

\begin{enumerate}
\item Initialization step: \\
 Compute the DCT-II-vector $\hat{\mathbf f}= (\hat{f}_{k})_{k=0}^{N-1}$ with $\hat{f}_{k} = \sum\limits_{\ell=0}^{N-1} f_{\ell} \, \cos \left( \frac{\pi (2\ell+1)k}{2N}  \right)$ of ${\mathbf f}$ and put
$$ \textstyle g_{k}:= (-1)^{k} 
 \left( \cos\left( \frac{\pi k}{2N}  \right) \right)^{-1} \widehat{f}_k, \qquad k=0, \ldots , N-1.$$ 
Set $z_{k}:=\cos(\frac{\pi k}{N})$ for $k =0, \ldots , N-1$.
Set $\mathbf \Gamma \coloneqq \left( k \right)_{k=0}^{N-1}$, ${\mathbf g}_{\mathbf \Gamma} \coloneqq (g_{k})_{k=0}^{N-1}$, ${\mathbf S}\coloneqq []$; ${\mathbf g}_{\mathbf S}\coloneqq []$, ${\mathbf r} \coloneqq (r_{k})_{k=0}^{N-1} \coloneqq {\mathbf 0}$. 
\item Preconditioning step:  \\
for $j=1:$ \textit{jmax}
\begin{itemize}
\item Compute $k \coloneqq \argmax_{\ell \in \mathbf \Gamma} |r_{\ell} - g_{\ell}|$
and update ${\mathbf S} \coloneqq ({\mathbf S}^{T}, k)^{T}$, ${\mathbf g}_{\mathbf S} \coloneqq ({\mathbf g}_{\mathbf S}^{T}, g_{k})^{T}$, and delete $k$ in $\mathbf{\Gamma}$ and $g_{k}$ in ${\mathbf g}_{\mathbf \Gamma}$.
\item Build  $\mathbf{C}_{N-j,j}\!\coloneqq\!\left( \frac{1}{z_{\ell}-z_{k}} \right)_{\ell \in \mathbf{\Gamma}, k \in \mathbf{S}}$, ${\mathbf L}_{N-j,j}\!\coloneqq\! \left( \frac{g_{\ell} - g_{k}}{z_{\ell}-z_{k}} \right)_{\ell \in \mathbf{\Gamma}, k \in \mathbf{S}}$.
\item Compute the right singular vector $
\mathbf{w}_{\mathbf S}$  corresponding to the smallest singular value $\sigma_{j}$ of $\mathbf{L}_{N-j,j}$. If $\sigma_{j} < tol \, \sigma_{1}$ (where $\sigma_{1}$ is the largest singular value of ${\mathbf L}_{N-j,j}$), then 
set $ M \coloneqq j-1$; 
delete the last entry of ${\mathbf S}$ and add it to ${\mathbf \Gamma}$;
stop (the last step for exact sampling data only).
 \item Compute $\mathbf{p} \coloneqq \mathbf{C}_{N-j,j} ({\mathbf  w}_{\mathbf S}.* \mathbf{g}_{\mathbf{S}})$, $\mathbf{q} \coloneqq \mathbf{C}_{N-j,j} {\mathbf w}_{\mathbf S}$ and $\mathbf{r} \coloneqq \mathbf{p}./\mathbf{q}$, where $.*$ denotes componentwise multiplication and $./$ componentwise division.
\end{itemize}
end (for)\\

\item Build the Loewner matrices  
$${\mathbf L}_{N-{M},{M}}^{(0)} = \left( \frac{g_{\ell} - g_{k} }{z_{\ell} - z_{k}} \right)_{\ell \in {\mathbf \Gamma}, k \in {\mathbf S}} , \quad
{\mathbf L}_{N-{M},{M}}^{(1)} = \left( \frac{g_{\ell}  \cos(\frac{\pi \ell}{N}) - g_{k} \cos(\frac{\pi k}{N}) }{z_{\ell} - z_{k}} \right)_{\ell \in {\mathbf \Gamma}, k \in {\mathbf S}},$$
and the joint matrix ${\mathbf L}_{N-{M},2{M}} \coloneqq \left(  {\mathbf L}^{(0)}_{N-{M},{M}} , {\mathbf L}^{(1)}_{N-{M},{M}} \right) \in {\mathbb C}^{N-{M},2{M}}$.

\item
Compute the SVD ${\mathbf L}_{N-{M},2{M}} = {\mathbf U}_{N-{M}} \, {\mathbf D}_{N-{M},2{M}} \, {\mathbf W}_{2{M}}$.\\ 
 Determine the vector of eigenvalues ${\mathbf b} =(b_{1}, \ldots , b_{M})^{T}$ of the matrix pencil \\
 $ z {\mathbf W}_{2M}(1:2M, 1:M) - {\mathbf W}_{2M}(1:2M, M+1:2M)$, 
 or equivalently the vector of eigenvalues of 
$$\left(  {\mathbf W}_{2M}(1:2M, 1:M) \right)^{\dagger}  {\mathbf W}_{2M}(1:2M, M+1:2M), $$
where $\Big({\mathbf W}_{2M}(1:2 M, 1:M)\Big)^{\dagger}$ denotes the Moore-Penrose inverse of  ${\mathbf W}_{2M}(1:2M, 1:M) $.
\item Compute ${\mathbf a} = (a_{j})_{j=1}^{M}$  as the least squares solution of the linear system 
$ \tilde{\mathbf C}_{N,M}  \, {\mathbf a} = {\mathbf g}, $
where $\tilde{\mathbf C}_{N,M} := \left( \frac{1}{z_{k} - b_{j}} \right)_{k=0,j=1}^{N-1,M}$ and ${\mathbf g} =(g_{k})_{k=0}^{N-1}$.
\item  Set $\phi_{j} = \frac{\arccos b_{j}}{h}$ and $\gamma_{j} = \frac{a_{j}}{\sin \frac{\phi_{j} h}{2} \, \sin( \phi_{j} hN)}$ for $j=1, \ldots , M$.
\end{enumerate}

\noindent
\textbf{Output:} $M \in {\mathbb N}$,  $\phi_{j}, \, \gamma_{j} \in {\mathbb R}$, $j=1, \ldots , M$.}
\end{algorithm}

\begin{remark}\label{remvan}
Instead of solving $\tilde{\mathbf C}_{N,M}  \, {\mathbf a} = {\mathbf g}$ in  step 5 of Algorithm \ref{alg5}
we could also directly solve
${\mathbf V}_{N,M} \, {\bgamma} = {\mathbf f}$ as in step 3 of Algorithm  \ref{algESPRIT}.
Indeed, multiplying the system ${\mathbf V}_{N,M} \, {\bgamma} = {\mathbf f}$ with $\textrm{diag} \big((-1)^{k} \big( \cos \big(\frac{\pi k}{2 N} \big)\big)^{-1}  \big)_{k=0}^{N-1}   {\mathbf C}_{N}^{II} $ from the left, where ${\mathbf C}_{N}^{II}$ is the cosine matrix defined by (\ref{CC}), we get
\begin{equation}\label{equi} \textstyle
\textrm{diag} \left((-1)^{k} \big(\cos \big(\frac{\pi k}{2 N}\big) \big)^{-1}  \right)_{k=0}^{N-1}   {\mathbf C}_{N}^{II} \,  {\mathbf V}_{N,M} \, {\bgamma} = \textrm{diag} \left((-1)^{k} \big( \cos \big(\frac{\pi k}{2 N} \big)\big)^{-1}  \right)_{k=0}^{N-1}   {\mathbf C}_{N}^{II} \,  {\mathbf f}.
\end{equation}
Since $\hat{\mathbf f} = {\mathbf C}_{N}^{II} \, {\mathbf  f}$, we get for the right hand side of (\ref{equi})
\begin{align*} \textstyle
 \textrm{diag} \left((-1)^{k} \big(\cos \big( \frac{\pi k}{2 N} \big)\big)^{-1}  \right)_{k=0}^{N-1}   {\mathbf C}_{N}^{II} \,  {\mathbf f} &=   \textstyle \textrm{diag} \left((-1)^{k} \big( \cos \big(\frac{\pi k}{2 N} \big)\big)^{-1}  \right)_{k=0}^{N-1} \hat{\mathbf f}\\
 &=  \textstyle \left((-1)^{k} 
 \big( \cos\big( \frac{\pi k}{2N}  \big) \big)^{-1} \widehat{f}_k\right)_{k=0}^{N-1}.
\end{align*}
For the left hand side of (\ref{equi}), using a similar computation as for $\hat{\mathbf f}$ in Section  \ref{sec31}, we have
\begin{align*}
\textrm{diag} &  \textstyle \left((-1)^{k} \big(\cos \big(\frac{\pi k}{2 N} \big)\big)^{-1}  \right)_{k=0}^{N-1}   {\mathbf C}_{N}^{II} \,  {\mathbf V}_{N,M} \, {\bgamma}\\
 &= \textstyle \textrm{diag} \left((-1)^{k} \big(\cos \big(\frac{\pi k}{2 N}\big) \big)^{-1}  \right)_{k=0}^{N-1}  \left( (-1)^{k} \cos \big(\frac{\pi k}{2N}\big)   \frac{\sin (\phi_j hN) \sin \big(\frac{\phi_j h}{2}\big)}{\cos\big(\frac{k\pi}{N}\big)-\cos\big(\phi_j h\big)}  \right)_{k=0,j=1}^{N-1,M} \, {\bgamma} \\ 
 &= \textstyle \left( \frac{1}{ \cos\big(\frac{k\pi}{N}\big)-\cos\big(\phi_j h\big) } \right)_{k=0,j=1}^{N-1,M} \textrm{diag} \left(  \sin (\phi_j hN) \sin\big(\frac{\phi_j h}{2}\big) \right)_{j=1}^{M} \, {\bgamma} \\ 
 & =\textstyle \left( \frac{1}{ \cos\big(\frac{k\pi}{N}\big)-\cos\big(\phi_j h\big) } \right)_{k=0,j=1}^{N-1,M} \left( \gamma_j \sin (\phi_j hN) \sin\big(\frac{\phi_j h}{2}\big) \right)_{j=1}^{M},
\end{align*}
which proves the equivalency. 

\end{remark}
 
Algorithm \ref{alg5} employs in the preconditioning step the greedy strategy of the AAA algorithm to find a suitable partition of the index set. Afterwords, we need to solve the matrix pencil problem for the Loewner matrices. This is done similarly as in the ESPRIT algorithm \ref{algESPRIT} for Toeplitz$+$Hankel matrices. 
One important advantage of the ESPIRA-II algorithm \ref{alg5} is that we don't longer need to treat integer frequencies in a 
post-processing step but can determine them directly.
 Analogously as for ESPIRA-I,  we have for ESPIRA-II an overall computational cost of ${\mathcal O}(N (M^{3}+\log N))$, where the  computation of DCT-II($N$) requires ${\mathcal O}(N \log N)$ operations, the preconditioning step based on the AAA algorithm requires ${\mathcal O}(N M^{3})$ operations, and the SVD of the Loewner matrix of size  $(N-M, 2 M)$ requires ${\mathcal O}(N M^{2})$ operations.
 For comparison,  ESPRIT in Section 2 requires ${\mathcal O}(N L^{2})$ flops, and  good recovery results are only achieved with $L\approx N/2$, i.e. with complexity ${\mathcal O}(N^{3})$. This computational cost for ESPRIT can be reduced using truncated SVD.
 
 \section{Numerical experiments}
 In this section we compare the performance of all three algorithms, ESPIRA-I, ESPIRA-II, and ESPRIT, for reconstruction of  cosine sums from exact and noisy input data. Further, we apply these algorithms to approximate even Bessel functions by short cosine sums. All algorithms are implemented in \textsc{Matlab} and use IEEE standard floating point arithmetic with double precision. 
The \textsc{Matlab} implementations can be found on \texttt{http://na.math.uni-goettingen.de} under \texttt{software}.  

 We consider short cosine sums as in (\ref{1.1}),  
$$
f(t) = \sum_{j=1}^{M} \gamma_{j} \, \cos(\phi_{j} t),
$$
with $M \in {\mathbb N}$, $\gamma_{j} \in {\mathbb R}\setminus \{ 0\}$, and  the pairwise distinct frequency parameters $\phi_{j}  \in [0, K)$ (with some $K>0$ be given). By $\tilde{f}$, $\tilde{\phi}_j$ and $\tilde{\gamma}_j$ we denote exponential sum, frequencies and coefficients respectively, reconstructed by our algorithms and 
define the relative errors by the following formulas
$$
e(f) \coloneqq \frac{\max|f(t)-\tilde{f}(t)|}{\max|f(t)|}, \quad  e(\bphi) \coloneqq  \frac{\max\limits_{j=1,\ldots,M}|\phi_j-\tilde{\phi}_j|}{\max\limits_{j=1,\ldots,M}|\phi_j|} 
\quad\text{and}\quad e(\gamra) \coloneqq \frac{\max\limits_{j=1,\ldots,M}|\gamma_j-\tilde{\gamma}_j|}{\max\limits_{j=1,\ldots,M}|\gamma_j|},
$$
where $e(f)$
be  the relative error of the exponential sum, such that the maximum is taken over the equidistant points in $[0, \frac{\pi N}{K}]$ with the step size $0.001$,  $e(\bphi)$ and $e(\gamra)$ be
 the relative errors for the  frequencies $\phi_j$ and the coefficients $\gamma_j$ respectively.
 
 \subsection{Reconstruction of exact data}
 First, we consider one example for reconstruction of cosine sums from exact data. In order to get optimal recovery results from the ESPRIT algorithm, we always take the upper bound $L$ for the order of exponential sums $M$ to be $L=N/2$, where $N$ is the number of given function samples. The number of terms $M$ is also recovered by each method.
 \begin{example}\label{ex1}
Let a signal be given by parameters $M=7$, $\gamma_j=j$ for $j=1,...,7$ and
$$
\bphi=\left(\sqrt{20}, \, \sqrt{0.2}, \, \sqrt{5}, \, \sqrt{15}, \, \sqrt{3}, \, \sqrt{15.1}, \, \sqrt{7}  \right).
$$
We reconstruct the signal from $N$ samples $f_{\ell}= \textstyle f\left(\frac{h (2\ell+1)}{2}\right)$, $\ell=0,\ldots,N-1$, with $h=\frac{\pi}{K}$ for different values of $N$ and $K$ (see Table \ref{tb1}), but each time we consider the segment $[0,5\pi]$. We take  $\varepsilon=10^{-10}$ for  ESPRIT and $tol=10^{-13}$ for ESPIRA-I and -II. We see that  for this example all algorithms work equally well.

\begin{table}[h!]
\centering
\caption{\small {Results of Example \ref{ex1}.}}
\label{tb1}
\begin{tabular}{ |p{0.5cm}|p{0.5cm}||p{2cm}|p{2cm}|p{2cm}|  }
 \hline
$N$ & $K$ & $e(f)$  & $e(\bphi)$ & $e(\gamra)$ \\
 \hline
 \multicolumn{5}{|c|}{ESPIRA-I} \\
 \hline
100  & 20  & $1.38\cdot 10^{-14}$    & $6.43\cdot 10^{-13}$ &   $3.08\cdot 10^{-13}$\\
150 & 30  &$1.19\cdot 10^{-13}$  & $3.48\cdot 10^{-11}$  &$3.66\cdot 10^{-12}$\\
200 & 40  &$3.97\cdot 10^{-13}$  & $1.56\cdot 10^{-10}$  &$7.79\cdot 10^{-11}$\\
 \hline
 \multicolumn{5}{|c|}{ESPIRA-II} \\
 \hline
 100   & 20 & $2.88\cdot 10^{-14}$ & $3.64\cdot 10^{-12}$&  $1.82\cdot 10^{-12}$\\
150 & 30  &$3.59\cdot 10^{-14}$  & $7.12\cdot 10^{-12}$  &$3.67\cdot 10^{-12}$\\
200 & 40 &   $4.86\cdot 10^{-14}$  & $7.47\cdot 10^{-12}$&$3.66\cdot 10^{-12}$\\
  \hline
 \multicolumn{5}{|c|}{ESPRIT} \\
 \hline
100  & 20  &$2.88\cdot 10^{-14}$    &$6.66\cdot 10^{-14}$&   $9.73\cdot 10^{-14}$\\
150 & 30   &$3.29\cdot 10^{-14}$  & $9.28\cdot 10^{-13}$  &$4.64\cdot 10^{-13}$\\
200 & 40   &$6.23\cdot 10^{-14}$  & $2.72\cdot 10^{-12}$  &$1.36\cdot 10^{-12}$\\
 \hline
\end{tabular}
\end{table}
\end{example}

\subsection{Reconstruction of noisy data}

In this subsection we consider reconstruction of noisy data, where we assume that the given signal values $f_k$ from (\ref{1.1}) are  corrupted with additive noise, i.e., the measurements are of the form 
$$ y_{k} = \textstyle f_k + \epsilon_{k} = f\Big(\frac{h(2k+1)\pi}{2} \Big) + \epsilon_{k}, \qquad k=0, \ldots , N-1,$$
where  $\epsilon_{k}$ are assumed to be from a uniform distribution with mean value $0$.  We will compare the performance of ESPRIT and ESPIRA-II. 

\begin{example}\label{ex3}

Let $f$ be given as in Example \ref{ex1}. We employ $N$ noisy signal values $y_{k}=f_k +\epsilon_{k}$, $k=0, \ldots, N-1$, where the random noise
$(\epsilon_{k})_{k=0}^{N-1}$ is generated in \textsc{Matlab} by 
\texttt{20*(rand(N,1)-0.5)}, i.e., we take uniform noise in the interval $[-10,10]$ with SNR (signal-to-noise ratio) and PSNR (peak signal-to-noise ratio) values as in Table \ref{tab_snr}.
We employ $N=1600$ or $N=2000$  equidistant measurements  with stepsize  $h=\frac{\pi}{50}$, i.e., $K=50$. In the first case  we use values on the interval $[0, 32\pi)$, in the second case on $[0, 40\pi)$.
We compute $10$ iterations and give the reconstruction errors for ESPRIT and for ESPIRA-II in Table \ref{tab_noise1}.
In the two algorithms ESPRIT and ESPIRA-II we assume that $M=7$ is known beforehand and modify the algorithms accordingly. For ESPRIT we compute the SVD decomposition in step 1 of Algorithm \ref{algESPRIT} but do not compute the numerical rank of ${\mathbf M}_{N-L+2, L}$ and instead take the known $M=7$ in step 2. For ESPIRA-II we set $jmax=M+1=8$ in Algorithm \ref{alg5}.
Further, for ESPIRA-II we use only the data $g_{k}$, $k=0, \ldots , \frac{N}{2}-1$ in order to avoid amplification of the error by the factor $\cos(\frac{\pi k}{2N})^{-1}$ in  the definition of $g_{k}$ in Algorithm \ref{alg5}.

\begin{table}[h!]
\centering
 \caption{\small SNR and PSNR values of the given noisy data with additive i.i.d.\ noise drawn from uniform distribution in $[-10,10]$ in Example \ref{ex3}}
  \label{tab_snr}
  
\begin{tabular}{ |p{1.2cm}|p{1.2cm}| p{1.2cm}||p{1.2cm}|p{1.2cm}| p{1.2cm}| }
 \hline
    \multicolumn{3}{|c||}{SNR} & \multicolumn{3}{|c|}{PSNR}\\
   \hline
 min & max & average & min & max & average\\[2mm]
 \hline
\ 3.96   & \ 4.14 & \ 4.07 & 13.57    & 13.76 & 13.68 \\[2mm]
 \hline
 \end{tabular}
\end{table}

As it can be seen in Table \ref{tab_noise1}, the two algorithms ESPRIT and ESPIRA-II  provide almost equally good results (ESPIRA-II works slightly better). 
In Figure \ref{plotnoise1}, we present the reconstruction results of the original function from the noisy data for  ESPRIT  and for ESPIRA-II. We plot the graphics and compute the relative errors $e(f)$  in the segment $[0,10]$. 

\begin{table}[h!]
\centering
 \caption{\small Reconstruction error for noisy data with additive i.i.d.\ noise drawn from uniform distribution in $[-10,10]$ in Example \ref{ex3}.}
  \label{tab_noise1}
\begin{tabular}{ |p{1.8cm}||p{1.6cm}|p{1.6cm}| p{1.6cm}||p{1.6cm}|p{1.6cm}| p{1.6cm}| }
 \hline
   & \multicolumn{3}{|c||}{ESPRIT} & \multicolumn{3}{|c|}{ESPIRA-II}\\
   \hline
$N =1600$ & min & max & average & min & max & average\\[2mm]
 \hline
$e(\bphi)$  & $1.11$e--$00$    &$9.94$e--$00$ & $5.49$e--$00$ & $1.31$e--$01$    &$4.19$e--$00$ & $8.67$e--$01$ \\[2mm]
$e(\gamra)$ & $3.12$e--$01$    & $4.16$e--$01$ & $3.57$e--$01$ & $1.85$e--$01$ & $3.72$e--$01$  & $2.98$e--$01$\\[2mm]
$e(f)$ & $1.37$e--$01$    & $2.37$e--$01$ & $1.73$e--$01$ & $7.24$e--$02$ & $1.26$e--$01$  & $9.83$e--$02$\\[2mm]
 \hline
 \end{tabular}
 \begin{tabular}{ |p{1.8cm}||p{1.6cm}|p{1.6cm}| p{1.6cm}||p{1.6cm}|p{1.6cm}| p{1.6cm}| }
 \hline
   & \multicolumn{3}{|c||}{ESPRIT} & \multicolumn{3}{|c|}{ESPIRA-II}\\
   \hline
$N =2000$ & min & max & average & min & max & average\\[2mm]
 \hline
$e(\bphi)$  & $2.74$e--$01$    &$9.28$e--$00$ & $5.23$e--$00$ & $4.24$e--$04$    &$3.65$e--$01$ & $2.28$e--$01$ \\[2mm]
$e(\gamra)$ & $2.47$e--$01$    & $3.47$e--$01$ & $3.01$e--$01$ & $3.32$e--$02$ & $3.32$e--$01$  & $2.51$e--$01$\\[2mm]
$e(f)$ & $1.33$e--$01$    & $1.91$e--$01$ & $1.68$e--$01$ & $3.68$e--$02$ & $1.40$e--$01$  & $1.01$e--$01$\\[2mm]
 \hline
 \end{tabular}
\end{table}

\begin{figure}[h!]
\includegraphics[width=7.0cm]{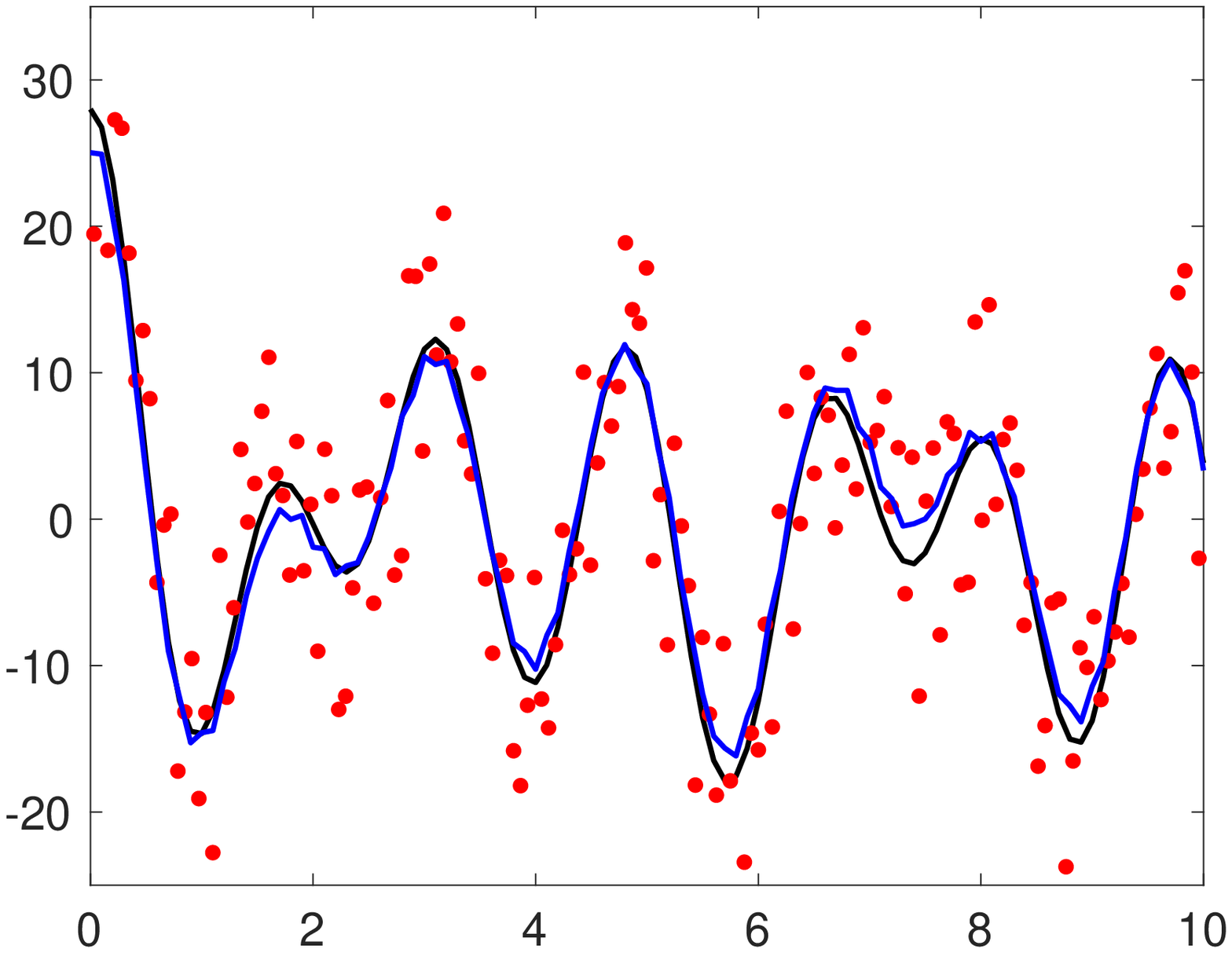}
\includegraphics[width=7.0cm]{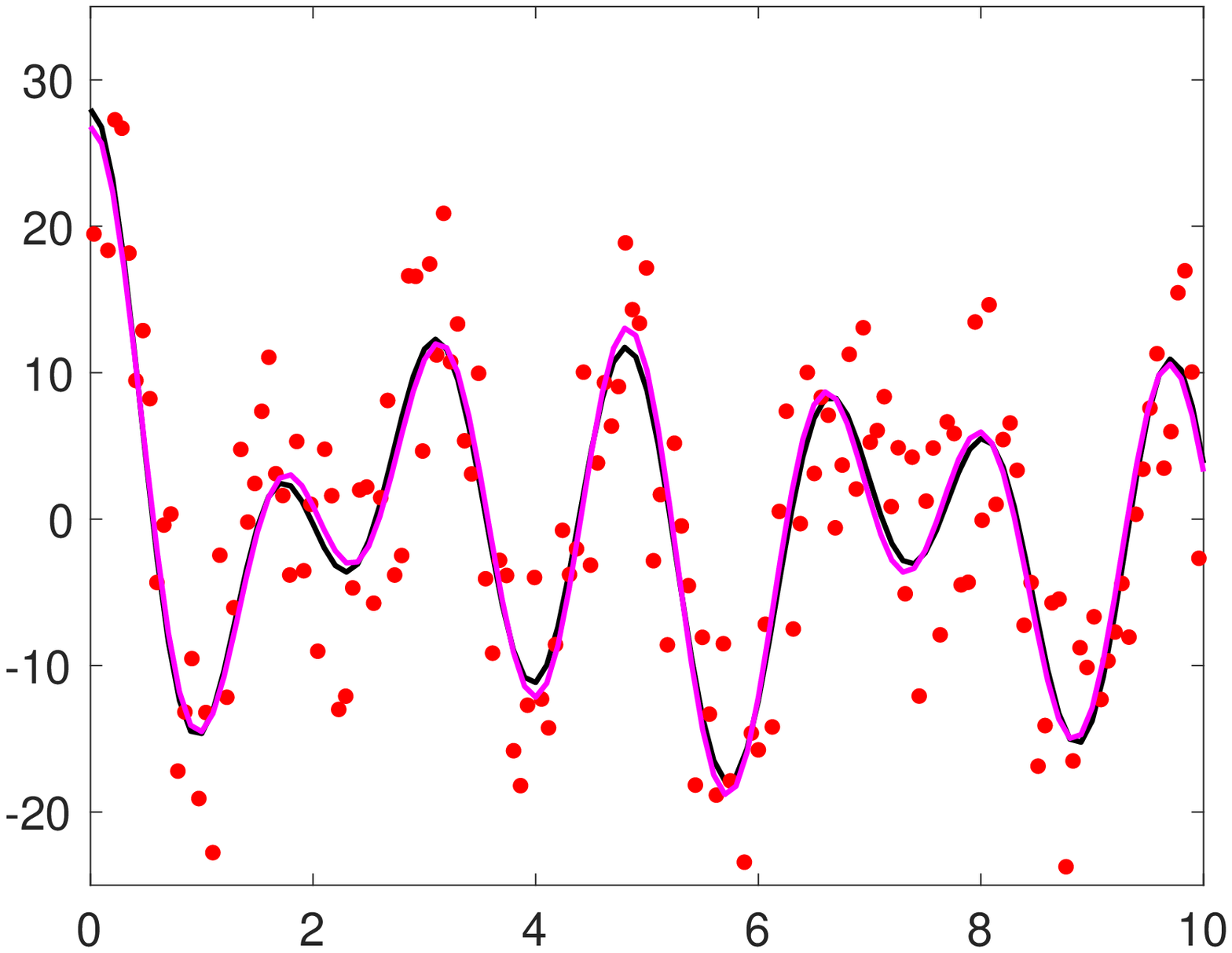}
\centering
\caption{\small Plot of the original function $f(t)$ (black), the given noisy data (red dots), and the achieved reconstructions restricted to the interval $[0, 10]$. Left: reconstruction by ESPRIT (blue). Right: reconstruction by ESPIRA-II (magenta). }
\label{plotnoise1}
\end{figure}
\end{example}

Similarly good reconstruction results are obtained from noisy input data for the case where $\epsilon_{k}$ are i.i.d.\ random  variables drawn from a standard normal distribution with mean value $0$.

\subsection{Application of ESPIRA to approximation of the Bessel function} 

We consider one example for approximation of an even Bessel function by short cosine sums by three methods ESPIRA-I and -II and ESPRIT. 
The algorithms are implemented in \textsc{Matlab} and use IEEE standard floating point arithmetic with double precision. 
 
By $J_n(t)$, $t \in \rr$, we denote the Bessel function  of the first kind of order $n \in \nn$. The power series expansion of this function is given by
$$
J_{n}(t)=\sum\limits_{k=0}^{\infty} \frac{(-1)^{k}}{k! \, (n+k)!} \left(\frac{t}{2}\right)^{2k+n}, \ t\geq 0.
$$
Similarly as in \cite{Cu2020},  we define a modified Bessel function on the interval $[0,\, B]$ by
\begin{equation}\label{besmod}
J_n(B,t)=\frac{B}{t} \, J_{n}(t), \ \ 0<t\leq B.
\end{equation}
Then $J_{2m+1}(B,t)$, $m=0,1,2,\ldots$, are even functions and therefore can be efficiently approximated by cosine sums. 

First we take $m=1$ and $B=126$ and consider the approximation error
\begin{equation}\label{errbes}
\textstyle \max\limits_{t \in [0, 126]}\left| J_3\left(126, t \right)-\sum\limits_{j=1}^{25} \gamma_j \cos(\phi_j t)  \right|.
\end{equation}

\begin{figure}[h]
\includegraphics[width=7.0cm]{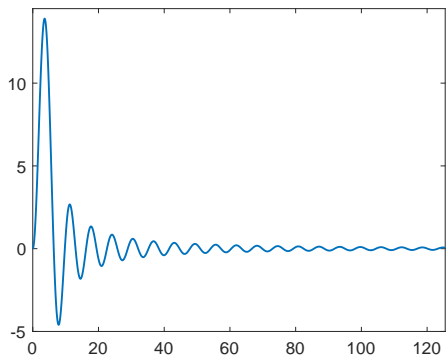} 
\centering
\caption{\small Bessel function $J_{3}(B,t)$ on the interval $[0,B]$ with $B=126$. }
\label{bes}
\end{figure}

We apply $N=400$ values of the function $J_{3}\left(126,t \right)$ at points $t_\ell =\frac{h(2 \ell+1)}{2 }$ for $\ell=0,1,\ldots,399$, with $h=\frac{\pi}{10}$, i.e., $t_{\ell} \in (0, 40 \pi) \subset (0, 126)$. 
For ESPRIT we use $L=200$. 
The maximum approximation error in (\ref{errbes}) provided by ESPIRA-I, ESPRIT and ESPIRA-II  are $1.18\cdot 10^{-6}$,  $1.78\cdot 10^{-6}$, and $4.28\cdot 10^{-6}$  respectively (see Fig. \ref{fig3}).

\begin{figure}[h]
\centering
\begin{minipage}[h]{0.45\linewidth}
\includegraphics[width=1.1\linewidth]{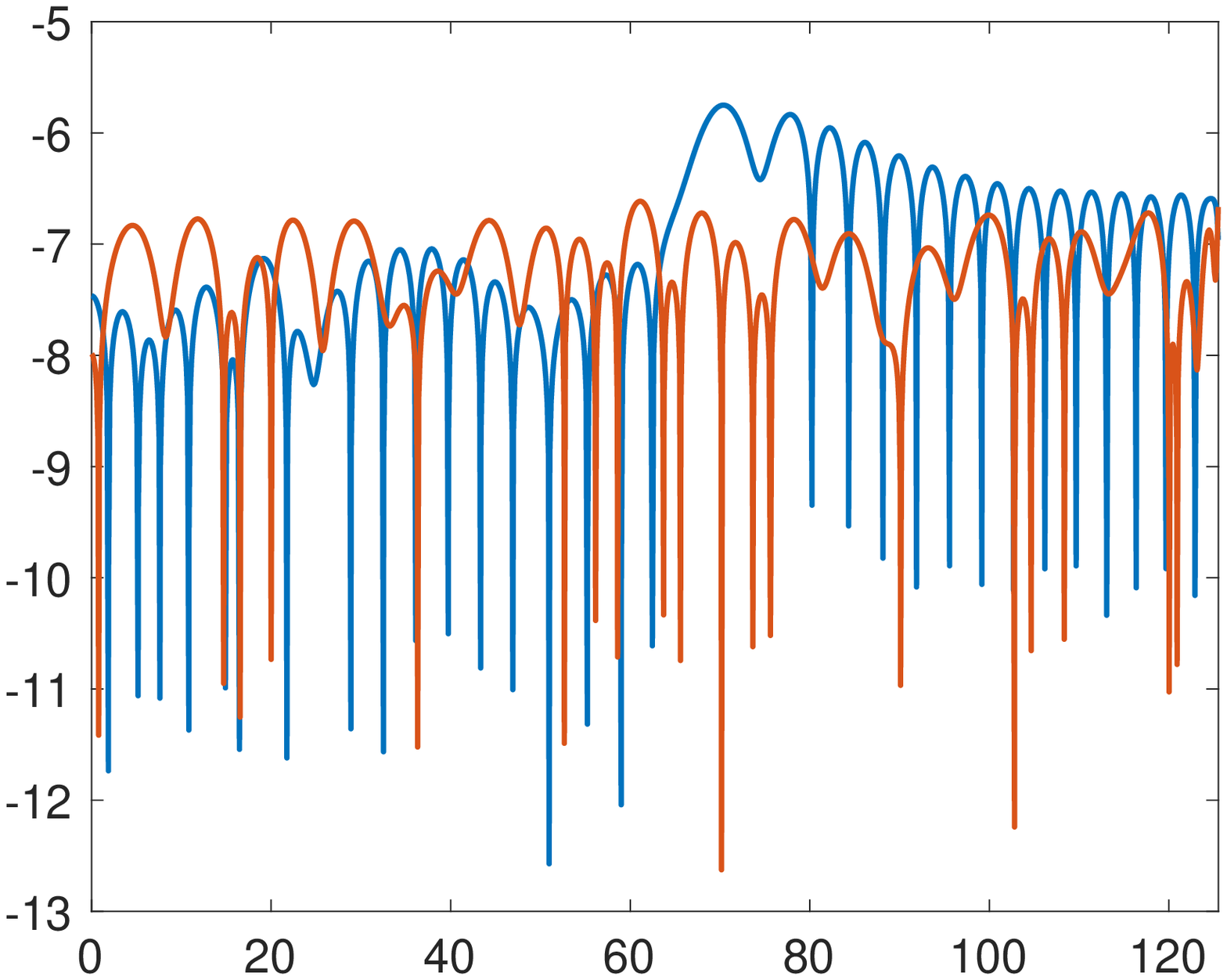}
\end{minipage}
\hfill
\begin{minipage}[h]{0.45\linewidth}
\includegraphics[width=1.1\linewidth]{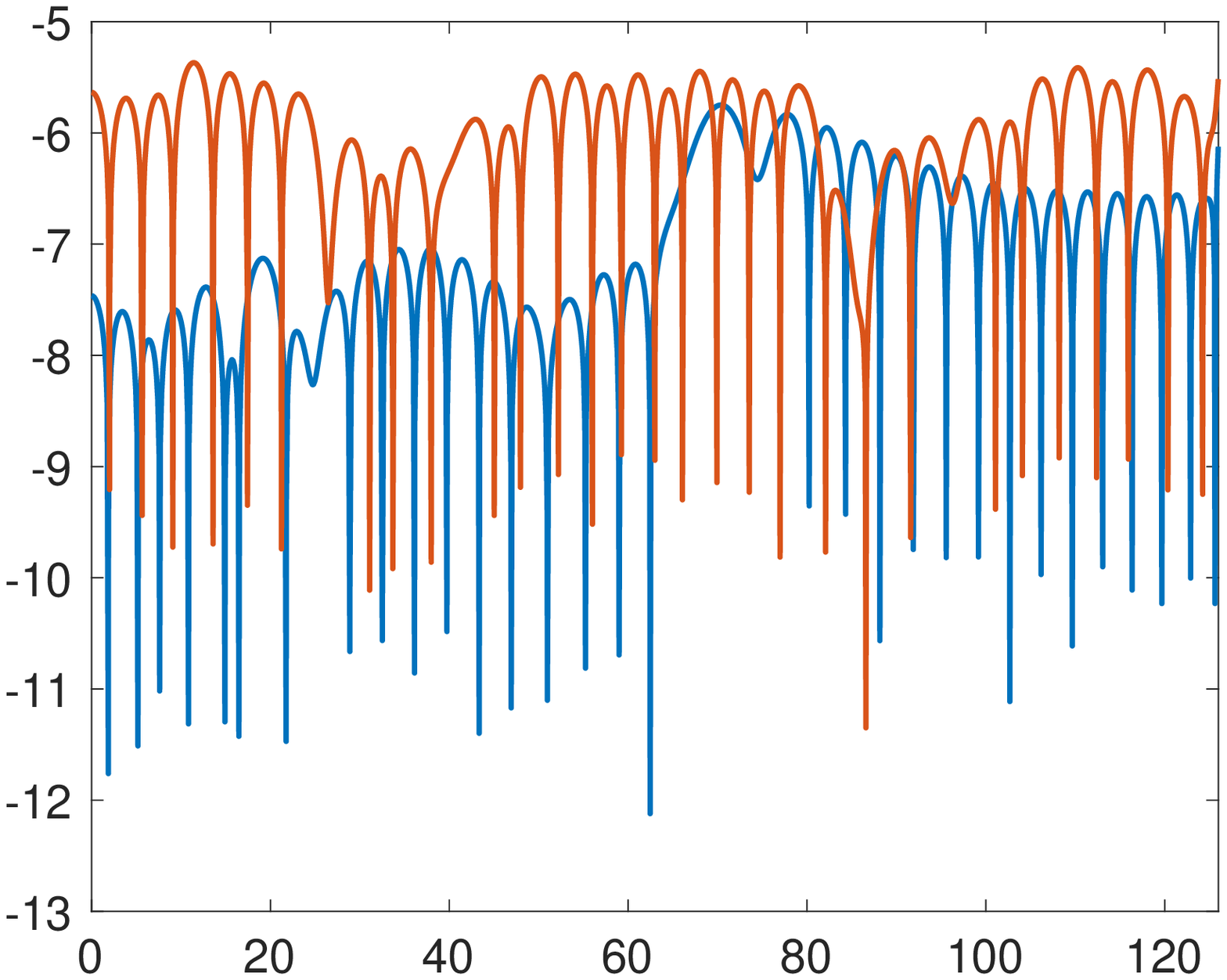}
\end{minipage}
\caption{\small The approximation error in (\ref{errbes}) in logarithmic scale obtained using a cosine sum with $M=25$ terms. Left: with ESPIRA-I (red) and ESPRIT (blue),
Right: with  ESPIRA-II (red) and ESPRIT (blue).}
\label{fig3}
\end{figure} 

In \cite{Cu2020}, high precision arithmetics  with 800 to 2200 digits  has been amployed  to find  approximations of $J_{0}(B,t)$ by sparse cosine sums for $B=1,5, 20$.  Our results  have  been obtained in double precision arithmetics for $B=126$ and are therefore not comparable.
However,  similarly as in \cite{Cu2020} we observe that the found frequencies $\phi_j$, $j=1,\ldots,25$, in (\ref{errbes}) have special locations. They are real and lie in the interval $[0,1]$, and their distribution becomes denser towards the point 1, see  Figure \ref{besfreq}.

\begin{figure}[h]
\includegraphics[width=0.8\linewidth]{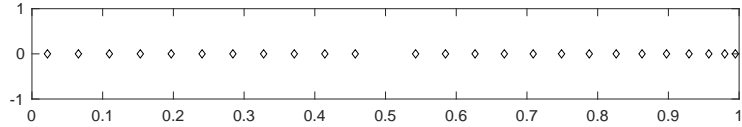} 
\centering
\caption{\small Distribution of the frequencies $\phi_j$, $j=1,\ldots,25$ in (\ref{errbes}) computed with ESPIRA-I. }
\label{besfreq}
\end{figure}

\section{Conclusions}

In this paper we have proposed three algorithms, ESPRIT, ESPIRA-I and ESPIRA-II, for stable recovery or approximation of functions by  cosine sums of the form
 $$
 f(t) = \sum_{j=1}^{M} \gamma_{j} \, \cos(\phi_{j} t),
 $$ 
where $M \in {\mathbb N}$, $\gamma_{j} \in {\mathbb R}\setminus \{ 0\}$, and  where $\phi_{j}  \in [0, K)$  are pairwise distinct.  We use  $N$ samples $f_{k} = f(\frac{\pi(2k +1)}{2K})$, $k=0, \ldots, N-1$, with $N > 2M$ to compute all parameters of this function model.\\
 Our new ESPRIT algorithm differs from similar methods considered in \cite{KP21} and \cite{PT2014}. We show that the values $\cos(\phi_j h)$ can be found as eigenvalues of a matrix pencil problem with Toeplitz+Hankel matrices. \\
 The new ESPIRA algorithm for cosine sums is based on the observation that the  DCT-II coefficients corresponding to the given vector of function values possesses a special rational structure if the frequencies satisfy 
$\phi_j \not\in \frac{\pi}{hN}\zz$ for $j=1, \ldots , M$. Therefore, the problem of parameters estimation can be reformulated as a rational interpolation/approximation problem. We have applied the stable AAA algorithm for iterative rational approximation. Frequency parameters $\phi_j \in \frac{\pi}{hN}\zz$ are reconstructed in a post-processing step. \\
We have shown that the ESPIRA approach can be reformulated  
as  a matrix pencil problem applied to Loewner matrices. This observation leads to a new version of the algorithm, called ESPIRA-II,
which has the advantage that case study of frequencies $\phi_j $ (and a corresponding post-processing step) is no longer needed.

All three methods provide good reconstruction results for noisy input data and construction of cosine sums of good approximations for even Bessel functions. Our numerical experiments show that all  algorithms work almost equally good, while ESPIRA-II gives slightly better results for reconstruction of noisy data.

\section*{Acknowledgement}
The authors gratefully acknowledge support by the German Research Foundation in the framework of the RTG 2088. The second author acknowledges support by the EU MSCA-RISE-2020 project EXPOWER. 

\small

\begin{thebibliography}{10}

%
%

\bibitem{Beinert17}
R.~Beinert and G.~Plonka.
\newblock Sparse phase retrieval of one-dimensional signals by {P}rony's
  method.
\newblock {\em Frontiers of Applied Mathematics and Statistics}, 3(5):open
  access, 2017.



\bibitem{BM05}
G.~Beylkin and L.~Monz\'{o}n.
\newblock On approximation of functions by exponential sums.
\newblock {\em Appl. Comput. Harmon. Anal.}, 19:17--48, 2005.


%
%
%
%
  \bibitem{Boyd13}
J. P. Boyd. 
\newblock A comparison of companion matrix methods to find roots of a trigonometric polynomial.
\newblock {\em 
J. Comput. Physics},  246:96-112, 2013.


\bibitem{Bre1}
C.~Brezinski and M.~Redivo-Zaglia.
Extrapolation Methods: Theory and Practice, North-Holland, Amsterdam, 1991.

\bibitem{Bre2}
C.~Brezinski and A.C.~Matos.
\newblock A derivation of extrapolation algorithms based on error estimates.
\newblock {\em J. Comput. Appl. Math.}, 66:5-26, 1996.

\bibitem{Bre3}
C.~Brezinski and M.~Redivo-Zaglia.
New representations of Pad\'e, Pad\'e-type, and partial Pad\'e approximants.
\newblock {\em J. Comput. Appl. Math.}, 284:69-77, 2015.



\bibitem{Cu2020}
A.~Cuyt, W.-s. Lee, and M.~Wu.
\newblock High accuracy trigonometric approximations of the real {B}essel
  functions of the first kind.
\newblock {\em Comput. Math. and Math. Phys.}, 60:119--127, 2020.


\bibitem{DP21}
N.~Derevianko and G.~Plonka.
\newblock Exact reconstruction of extended exponential sums using rational
  approximation of their {F}ourier coefficients, {\em Analysis Appl.}, to appear, open access,  DOI: 10.1142/S0219530521500196.

\bibitem{DPP21}
N.~Derevianko,  G.~Plonka, and M.~Petz
\newblock From ESPRIT to ESPIRA: Estimation of Signal Parameters by Iterative Rational Approximation, {\em IMA J. Numer. Anal.}, to appear.
  
\bibitem{DHT14}
T.A. Driscoll, N.~Hale, L.N. Trefethen, and eds.
\newblock Chebfun user’s guide.
\newblock Pafnuty Publications, Oxford, see also www.chebfun.org, 2014.



%
%
%


\bibitem{Hua90}
Y.~Hua and T.K. Sarkar.
\newblock Matrix pencil method for estimating parameters of exponentially
  damped/undamped sinusoids in noise.
\newblock {\em IEEE Trans. Acoust. Speech Signal Process.}, 38:814--824, 1990.

%
%

\bibitem{Izat}
J.A. Izat and M.A. Choma, Theory of optical coherence tomography. In: Drexler, W., Fujimoto, J.G. (eds.) {\em Optical Coherence Tomography}. Biomedical and Medical Physics, Biomedical Engineering, Springer, p. 47--72, 2008.

\bibitem{KP21}
I.~Keller and G.~Plonka.
Modifications of Prony's method for the recovery and sparse approximation of generalized exponential sums.
In: Fasshauer G.E., Neamtu M., Schumaker L.L. (eds) {\em Approximation Theory XVI. AT 2019}. 
Springer, Cham, p.~123--152, 2021.

\bibitem{Klein}
G.~Klein.
\newblock {\em Applications of Linear Barycentric Rational Interpolation}.
\newblock PhD thesis Fribourg, Switzerland, 2012.

%

\bibitem{NST18}
Y.~Nakatsukasa, O.~Sete, and L.N. Trefethen.
\newblock The {AAA} algorithm for rational approximation.
\newblock {\em SIAM J. Sci. Comput.}, 40(3):A1494--A1522, 2018.

\bibitem{Osborne95}
M.R. Osborne and G.K. Smyth.
\newblock A modified {P}rony algorithm for exponential function fitting.
\newblock {\em SIAM J. Sci. Comput.}, 16(1):119--138, 1995.

\bibitem{Pereyra10}
V.~Pereyra and G.~Scherer.
\newblock Exponential data fitting.
\newblock In {\em Exponential Data Fitting and its Applications}, pages 1--26.
  Bentham Sci. Publ, 2010.
  

\bibitem{PPD21}
M.~Petz, G.~Plonka, and N.~Derevianko.
\newblock Exact reconstruction of sparse non-harmonic signals from their Fourier coefficients.
\newblock {\em Sampl. Theory Signal Process. Data Anal.}, 19(7), 2021. https://doi.org/10.1007/s43670-021-00007-1, open access.


\bibitem{PPST18}
G.~Plonka, D.~Potts, G.~Steidl, and M.~Tasche.
\newblock {\em Numerical {F}ourier {A}nalysis}.
\newblock Birkh\"auser, Basel, 2018.

\bibitem{PSK19}
G.~Plonka, K.~Stampfer, I.~Keller. Reconstruction of stationary and non- stationary signals by the generalized Prony method. 
\newblock {\em Anal. and Appl.}, 17(2):179--210, 2019.


\bibitem{PT05}
G.~Plonka and M.~Tasche.
Fast and numerically stable algorithms for discrete cosine transforms.
\newblock {\em Linear Algebra  Appl.}, 394:309--345, (2005). 

\bibitem{PT14}
G.~Plonka and M.~Tasche.
\newblock Prony methods for recovery of structured functions.
\newblock {\em GAMM Mitt.}, 37(2):239--258, 2014.

\bibitem{PT2013}
D.~Potts and M.~Tasche.
\newblock Parameter estimation for nonincreasing exponential sums by
  {P}rony-like methods.
\newblock {\em Linear Algebra Appl.}, 439(4):1024--1039, 2013.

\bibitem{PT2014}
D.~Potts and M.~Tasche.
Sparse polynomial interpolation in Chebyshev bases. 
\newblock {\em Linear Algebra Appl.}, 441:61--87, 2014.

\bibitem{PT15}
D.~Potts and M.~Tasche.
\newblock Fast {ESPRIT} algorithms based on partial singular value
  decompositions.
\newblock {\em Appl. Numer. Math.}, 88:31--45, 2015.

\bibitem{RK89}
R.~Roy and T.~Kailath.
\newblock {ESPRIT} estimation of signal parameters via rotational invariance
  techniques.
\newblock {\em IEEE Trans. Acoust. Speech Signal Process.}, 37:984--995, 1989.

%
%
\bibitem{SP20}
K.~Stampfer and G.~Plonka.
The generalized operator-based Prony method.
\newblock {\em Constr. Approx.} 52:247--282, 2020. 

\bibitem{VMB02}
M.~Vetterli, P.~Marziliano, and T.~Blu.
\newblock Sampling signals with finite rate of innovation.
\newblock {\em IEEE Trans. Signal Process.}, 50(6):1417--1428, 2002.


\bibitem{ZP19}
R.~Zhang and G.~Plonka.
\newblock Optimal approximation with exponential sums by a maximum likelihood
  modification of {P}rony’s method.
\newblock {\em Adv. Comput. Math.}, 45(3):1657--1687, 2019.

\end{thebibliography}

\end{document}